\documentclass[12pt]{article}
\usepackage[margin=1.0in]{geometry}
\usepackage{graphicx}
\usepackage{amsmath}
\usepackage{amsthm}
\usepackage{amssymb}
\usepackage{epstopdf}
\usepackage{color}
\usepackage{authblk}
\usepackage{dirtytalk}
\usepackage{pgfplots}
\usepackage{tikz}
\usepackage{longtable}

\newtheorem{theorem}{Theorem}[section]
\newtheorem{lemma}[theorem]{Lemma}
\newtheorem{claim}[theorem]{Claim}

\newtheorem{conjecture}[theorem]{Conjecture}

\newtheorem{quest}[theorem]{Question}

\def\lc{\left\lceil}   
\def\rc{\right\rceil}
\def\lf{\left\lfloor}   
\def\rf{\right\rfloor}

\pgfplotsset{compat=1.18}

\usepackage{setspace}

\begin{document}
\title{Counting independent sets in regular graphs with bounded independence number}
\author{David Galvin\thanks{Department of Mathematics, University of Notre Dame, Notre Dame IN 46556; dgalvin1@nd.edu. Supported in part by the Simons Foundation.} \\ Phillip Marmorino\thanks{Department of Mathematics, University of Notre Dame, Notre Dame IN 46556; pmarmori@nd.edu.}}

\maketitle

\begin{abstract}
An $n$-vertex, $d$-regular graph can have at most $2^{n/2+o_d(n)}$ independent sets. In this paper we address what happens with this upper bound when we impose the further condition that the graph has independence number at most $\alpha$. 

We give upper and lower bounds that in many cases are close to each other. In particular, for each $0 < c_{\rm ind} \leq 1/2$ we exhibit a constant $k(c_{\rm ind})$ such that
\begin{itemize}
\item If $(G_n)_{n \in {\mathbb N}}$ is a sequence of graphs with $G_n$ $d$-regular on $n$ vertices and with maximum independent set size at most $\alpha$, with $d\rightarrow \infty$ and $\alpha/n \rightarrow c_{\rm ind}$ as $n \rightarrow \infty$, then $G_n$ has at most $k(c_{\rm ind})^{n+o(n)}$ independent sets, and
\item there is a sequence $(G_n)_{n \in {\mathbb N}}$ of graphs with $G_n$ $d$-regular on $n$ vertices ($d \leq n/2$) and with maximum independent set size at most $\alpha$, with $\alpha/n \rightarrow c_{\rm ind}$ as $n \rightarrow \infty$ and with $G_n$ having at least $k(c_{\rm ind})^{n+o(n)}$ independent sets.
\end{itemize}
We also consider the regime $1/2 < c_{\rm ind} < 1$. Here for each $0 < c_{\rm deg} \leq 1-c_{\rm ind}$ we exhibit a constant $k(c_{\rm ind},c_{\rm deg})$ for which an analogous pair of statements can be proven, except that in each case we add the condition $d/n \rightarrow c_{\rm deg}$ as $n \rightarrow \infty$. 

Our upper bounds are based on graph container arguments, while our lower bounds are constructive. 
 \end{abstract}

\section{Introduction and statement of results}

Let $i(G)$ denote the number of independent sets (sets of mutually non-adjacent vertices) admitted by a graph $G$, and denote by $i_k(G)$ the number of independent sets of size $k$. Trivially $i(G) \leq 2^n$ for any $n$-vertex graph. There has been much work in recent years devoted to the extremal enumerative question of maximizing (or minimizing) $i(G)$ when $G$ is restricted to run over various classes of $n$-vertex graphs, such as trees \cite{ProdingerTichy}, graphs with a given number of vertices and edges \cite{CutlerRadcliffe}, graphs with a given minimum degree \cite{Chase, EngbersGalvin, GanLohSudakov, Galvin2}, and regular graphs with girth conditions \cite{PerarnauPerkins}. (This is very much a partial list.) 

\medskip

Our focus will be on the family of regular graphs, a family to which much attention has been given in the context of the extremal enumerative question. Say that a graph is an {\it $(n,d)$-graph} if it has $n$ vertices and is $d$-regular. Granville, while considering a problem from combinatorial group theory, conjectured that every $(n,d)$-graph satisfies 
\begin{equation} \label{Granville}
i(G) \leq 2^{n\left(1/2+o(1)\right)}
\end{equation}
where $o(1)\rightarrow 0$ as $d \rightarrow \infty$ (see \cite{Alon}). Note that any bipartite $(n,d)$-graph witnesses that this bound is tight up to the $o(1)$ term.

Granville's conjecture has generated much follow-up work. Alon \cite{Alon} proved the conjecture, showing that the $o(1)$ term in \eqref{Granville} can be taken to be $O\left(d^{-1/10}\right)$ and he further speculated that for an $(n,d)$-graph $G$ with $2d|n$ we have $i(G) \leq i(G_{n,d})$ where $G_{n,d}$ is the disjoint union of $n/2d$ copies of the complete bipartite graph $K_{d,d}$. Note that $$
i(G_{n,d})=(2^{n+1}-1)^{n/2d} = \exp_2\left\{\frac{n}{2}\left(1 + \frac{1+o(1)}{d}\right)\right\}
$$ 
with $o(1) \rightarrow 0$ as $d \rightarrow \infty$.

Using graph containers Sapozenko improved Alon's bound on the $o(1)$ term in \eqref{Granville} to $O\left(\sqrt{(\log d)/d}\right)$ \cite{Sapozhenko}. Around the same time Kahn \cite{Kahn} used entropy methods to establish $i(G) \leq i(G_{n,d})$ for all {\it bipartite} $(n,d)$-graphs with $2d|n$, and also (personal communication to the first author, see \cite{MadimanTetali} for a proof) he showed that for general (not necessarily bipartite) $(n,d)$-graphs the $o(1)$ term in \eqref{Granville} may be taken to be $2/d$, only a factor of $2$ away from optimal. 

By combining containers and entropy Galvin \cite{Galvin} improved Kahn's $2/d$ to $(1+o(1))/d$, and Galvin and Zhao \cite{GalvinZhao} showed that $i(G) \leq i(G_{n,d})$ for all $(n,d)$-graphs with $2d|n$ and $d \leq 5$. Finally Zhao \cite{Zhao} obtained the tight result that $i(G) \leq i(G_{n,d})$ for all $(n,d)$-graphs with $2d|n$, via a clever reduction to the bipartite case. Other proofs of this result have subsequently been found by Davies, Jenssen, Perkins and Roberts \cite{DaviesJenssenPerkinsRoberts}, by Lubetzky and Zhao \cite{LubetzkyZhao}, and (via a significant generalization to graphs that are not necessarily regular) by Sah, Sawhney, Stoner and Zhao \cite{SahSawhneyStonerZhao}.  

\medskip

A much older variant of the extremal enumerative question for independent sets involves putting an upper bound on the independence number of $G$. Denote by $Z(n, \alpha)$ the $n$-vertex graph consisting of a disjoint union of $\alpha$ cliques, with orders as near equal as possible. Equivalently $Z(n,\alpha)$ is the complement of the $n$-vertex, $\alpha$-class Tur\'an graph ($Z$ here stands for Zykov; the reason for this choice of notation will become clear in a moment). Tur\'an's theorem \cite{Turan} (applied to the complement of a graph) says that among graphs with $n$ vertices and with maximum independent set size at most $\alpha$, it is $Z(n, \alpha)$ that uniquely maximizes the count of independent sets of size $2$. Zykov \cite{Zykov} generalized Tur\'an's theorem, showing that among $n$ vertex graphs with maximum independent set size at most $\alpha$ it is $Z(n, \alpha)$ that uniquely maximizes the count of independent sets of size $k$, for any $k$. An immediate consequence of this is that if $G$ has $n$ vertices and maximum independent set size at most $\alpha$ then $i(G) \leq i(Z(n,\alpha))$.

\medskip

The main point of this paper is to consider what happens when the two restrictions discussed above --- regularity and upper bound on the size of the maximum independent set --- are combined. 

Say that a graph is an $(n,d,\alpha)$-graph if it is an $(n,d)$-graph and has largest independent size at most $\alpha$. To avoid trivialities we assume throughout that $d \geq 1$, in which case we may also assume $\alpha \leq n/2$ (since an $(n,d)$ graph with $d \geq 1$ has maximum independent set size at most $n/2$). If $\alpha=n/2$ we are placing no restriction beyond that $G$ is an $(n,d)$-graph, and so the results of Kahn and Zhao cited earlier settle the question of maximizing the count of independent sets in this case. (At least when $d \leq n/2$; see below for a discussion of the case $d > n/2$.) 

We also note that if $\alpha=1$ then there is a unique $(n,d,\alpha)$-graph when $d=n-1$ (the complete graph), and there are no $(n,d,\alpha)$-graphs for any other choices of $d$. In this case all extremal enumerative questions are trivial, and so from here on we assume $\alpha \geq 2$. By the same token (uniqueness or non-existence of $(n,d,\alpha)$-graphs) there is no loss in assuming $n \geq 6$ and $d \geq 2$.

All bipartite $(n,d)$-graphs with $n$ even (which witness that the largest number of independent sets admitted by an $(n,d)$-graph is at least $2^{n/2+o(n)}$) have an independent set of size exactly $n/2$, and this is a characterization: an $(n,d)$-graph (with $d \geq 1$ and $n$ even) has an independent set of size exactly $n/2$ if and only if it is bipartite. 

It is natural to suppose that if we impose an upper bound (smaller than $n/2$) on the size of the largest independent set admitted by an $(n,d)$-graph, then we will get a reduction in the maximum possible number of independent sets admitted. Indeed, it is tempting (though, as it turns out, incorrect) to conjecture that for $\alpha \leq n/2$ an $(n,d,\alpha)$-graph admits at most $2^{\alpha+o(n)}$ independent sets (the $2^\alpha$ term coming from arbitrary subsets of some maximum independent set). That this conjecture is not true was probably first observed by Daynyak \cite{Daynyak2007}, who showed that for each fixed $d \geq 3$ there is a sequence of $d$-regular graphs $(G_m)_{m \geq 1}$ with $|V(G_m)| \rightarrow \infty$ as $m \rightarrow \infty$, with $G_m$ having largest independent set size at most $|V(G_m)|(1-1/k^2)/2$, and with $G_m$ having at least $\exp_2\left\{|V(G_m)|(1+1/(2k))/2\right\}$ independent sets.

Sapozhenko \cite{Sapozhenko2} considered the question of bounding $i(G)$ for $G$ in the family of $(n,d,\alpha)$-graphs, and showed that there is an absolute constant $c$ such that if $G$ is an $(n,d,\alpha)$-graph, then
\begin{equation} \label{eq-Sap-alpha-ub}
i(G) \leq \left(1+\frac{n}{2\alpha}\right)^\alpha \exp_2\left\{cn\sqrt{\frac{\log d}{d}}\right\}.
\end{equation}
In particular if $\alpha=c_{\rm ind}n+o(n)$ ($c_{\rm ind} >0$ a constant) then we have from \eqref{eq-Sap-alpha-ub} that for an $(n,d,\alpha)$-graph
$$
i(G) \leq \left(1+\frac{1}{2c_{\rm ind}}\right)^{c_{\rm ind}n+o(n)}.
$$

One goal of this paper is to improve Sapozhenko's inequality \eqref{eq-Sap-alpha-ub}, and in fact to obtain a bound that in many cases is tight (up to the $o(n)$ term). Here we present one of our two main theorems.
\begin{theorem} \label{thm-ind-sets-alpha=cn-d<=n/2}
For each $0 < c_{\rm ind} \leq 1/2$ set
$$
k(c_{\rm ind}) =
    \begin{cases} 
      \left(1+\frac{1}{2c_{\rm ind}}\right)^{c_{\rm ind}} & \text{if } 1/(2c_{\rm ind}) \text{ is an integer}  \\
      \left(1+\lf\frac{1}{2c_{\rm ind}}\rf\right)^{c_{\rm ind}\lc\frac{1}{2c_{\rm ind}}\rc-1/2} \left(1+\lc\frac{1}{2c_{\rm ind}}\rc\right)^{1/2-c_{\rm ind}\lf\frac{1}{2c_{\rm ind}}\rf} & \text{otherwise.}
   \end{cases}
$$
\begin{description}
\item[UB1] Let $N \subseteq {\mathbb N}$ be an infinite set, and let $\{(n,d_n,\alpha_n):n \in N\}$ be a sequence of triples with
\begin{description}
\item[P1] $d_n \rightarrow \infty$ as $n \rightarrow \infty$ and
\item[P2] $\alpha_n/n \rightarrow c_{\rm ind}$ as $n \rightarrow \infty$.
\end{description}
If $(G_n)_{n \in N}$ is a sequence of graphs with $G_n$ an $(n, d_n, \alpha_n)$ graph for each $n \in N$, then
$$
i(G_n) \leq k(c_{\rm ind})^{n\left(1+o(1)\right)}
$$
where $o(1)\rightarrow 0$ as $n \rightarrow \infty$.
\item[LB1] There is 
\begin{itemize}
\item an infinite set $N \subseteq {\mathbb N}$, 
\item a sequence of triples $\{(n,d_n,\alpha_n):n \in N\}$ satisfying {\bf P1} and {\bf P2} as well as
\begin{description}
\item[P3] $\limsup_{n \rightarrow \infty} d_n/n \leq 1/2$,
\end{description}
and 
\item a sequence of graphs $(G_n)_{n \in N}$ with $G_n$ an $(n, d_n, \alpha_n)$ graph for each $n \in N$ 
\end{itemize}
such that
$$
i(G_n) \geq k(c_{\rm ind})^{n\left(1+o(1)\right)}
$$
where $o(1)\rightarrow 0$ as $n \rightarrow \infty$.
\end{description}
\end{theorem}
So if $\alpha$ scales linearly with $n$, the largest number of independent sets admitted by an $(n,d,\alpha)$-graph (with $d \leq n/2$) grows exponentially with $n$ with a precisely determinable base that depends on the scaling factor. Note that while Sapozhenko's bound \eqref{eq-Sap-alpha-ub} gives {\bf UB1} when $c_{\rm ind}$ is of the form $1/(2m)$ for integer $m$, for all other choices of $c_{\rm ind}$ we have $k(c_{\rm ind}) <(1+1/(2c_{\rm ind}))^{c_{\rm ind}}$, so {\bf UB1} strictly improves \eqref{eq-Sap-alpha-ub} in these cases. As a specific instance observe that \eqref{eq-Sap-alpha-ub} gives $i(G) \leq (5/2)^{n/3+o(n)} = (1.357\cdots)^{n+o(n)}$ for $G$ an $(n,d,n/3)$-graph, while ${\bf UB1}$ yields $i(G) \leq 6^{n/6+o(n)} = (1.348\cdots)^{n+o(n)}$.

Note that if $c_{\rm ind}=0$ (that is, if $\alpha_n = o(n)$) then we have an easy upper bound on the count of independent sets in an $(n,d,\alpha)$-graph $G$, namely
$$
i(G) \leq 2^{\alpha_n} \binom{n}{\alpha_n} = 2^{o(n)}
$$
(take arbitrary subsets of sets of size $\alpha_n$), that is best possible in the same sense that the bounds in Theorem \ref{thm-ind-sets-alpha=cn-d<=n/2} are best possible (up to $o(n)$ in the exponent).

\medskip

In {\bf LB1} we (essentially) restrict attention to $(n,d)$-graphs with $d \leq n/2$. There is a good reason for this. Although the bound {\bf UB1} is valid for all $(n,d,\alpha)$-graphs that satisfy {\bf P1} and {\bf P2}, it is only tight if we add condition {\bf P3}. Before exploring this further (see Theorem \ref{thm-ind-sets-alpha=cn-d=an-a>=1/2} below), we pause to make the more fundamental observation that while the (by now very familiar) bound $i(G) \leq 2^{n(1+o(1))/2}$ for $(n,d)$-graphs is valid for all choices of $n$ and $d$, for $d > n/2$ it is not tight. Indeed, we have the following simple observation.
\begin{claim} \label{clm-(n,d):d>n/2}
For $n, d \geq 1$ with $nd$ even and $n/2 < d < n$ we have
\begin{description}
\item[UB2] if $G$ is an $(n,d)$-graph then 
$$
i(G) \leq n2^{n-d} = 2^{n-d + o(n)}
$$
(where $o(1)\rightarrow 0$ as $n, d \rightarrow \infty$)
and
\item[LB2] there is an $(n,d)$-graph $G$ that has an independent set of size $n-d$ and so has
$$
i(G) \geq 2^{n-d}.
$$
\end{description}
\end{claim}

We give the (easy) proof here, as a very simple preview of our container argument for upper bounds and constructions for lower bounds.
\begin{proof} (Proof of Claim \ref{clm-(n,d):d>n/2}.)
We begin with {\bf UB2}. Let $I$ be a non-empty independent set in $G$, and let $v \in I$ be any vertex of $I$. Since no neighbour of $v$ is in $I$, we have $I \subseteq V(G)\setminus N(v)$, a set of size $n-d$. Thus by considering arbitrary non-empty subsets of $V(G)\setminus N(v)$, as $v$ runs over all vertices of $G$, we get an upper bound on the number of non-empty independent sets of $G$ of $n(2^{n-d}-1)$, so $i(G) \leq n(2^{n-d}-1) +1 \leq n2^{n-d}$.    

We now move on to {\bf LB2}. Note that $d > 2d-n$ (this is equivalent to $d < n$), and also that because $nd$ is even, we have that $d(2d-n)$ is even. These facts together imply that there is a $(2d-n)$-regular graph on any set of $d$ vertices. 

Construct a graph $G$ on vertex set $\{v_1, \ldots, v_{n-d}\} \cup \{w_1, \ldots, w_d\}$ as follows:
\begin{itemize}
\item $\{v_1, \ldots, v_{n-d}\}$ induces a subgraph with no edges.
\item For each $i=1,\ldots,n-d$ and $j=1,\ldots,d$, $v_i$ is adjacent to $w_j$.
\item The subgraph induced by $\{w_1, \ldots, w_d\}$ is $(2d-n)$ regular. 
\end{itemize}
It is easy to check that $G$ is $d$-regular and has an independent set of size $n-d$ (namely $\{v_1, \ldots, v_{n-d}\}$).
\end{proof}

The conclusion of Claim \ref{clm-(n,d):d>n/2} suggests that in the regime $d > n/2$, the analog of Theorem \ref{thm-ind-sets-alpha=cn-d<=n/2} will look somewhat different. In particular, the growth rate of the maximum number of independent sets in an $(n,d,\alpha)$ graph may (and in fact does) depend on $d$ as well as $\alpha$. In this direction we now present our second main theorem. Note that when $d > n/2$, the size of the largest independent set in an $(n,d)$-graph is at most $n-d$.
\begin{theorem} \label{thm-ind-sets-alpha=cn-d=an-a>=1/2}
For each $1/2 \leq c_{\rm deg} < 1$ and $0 < c_{\rm ind} \leq 1-c_{\rm deg}$ set
$$
k(c_{\rm ind},c_{\rm deg}) = \begin{cases}
    \left(1+\frac{1-c_{\rm deg}}{c_{\rm ind}}\right)^{c_{\rm ind}}  & \text{if } \frac{1-c_{\rm deg}}{c_{\rm ind}} \in {\mathbb N}\\
    \left(1+\lf \frac{1-c_{\rm deg}}{c_{\rm ind}} \rf \right)^{c_{\rm ind}\lc \frac{1-c_{\rm deg}}{c_{\rm ind}} \rc-1+c_{\rm deg}}  \left(1+\lc \frac{1-c_{\rm deg}}{c_{\rm ind}} \rc \right) ^{1-c_{\rm deg}+\lf \frac{1-c_{\rm deg}}{c_{\rm ind}}\rf} & \text{otherwise.}
\end{cases}
$$
\begin{description}
\item[UB3] Let $N \subseteq {\mathbb N}$ be an infinite set, and let $\{(n,d_n,\alpha_n):n \in N\}$ be a sequence of triples with
\begin{description}
\item[Q1] $d_n/n \rightarrow c_{\rm deg}$ as $n \rightarrow \infty$ and
\item[Q2] $\alpha_n/n \rightarrow c_{\rm ind}$ as $n \rightarrow \infty$.
\end{description}
If $(G_n)_{n \in N}$ is a sequence of graphs with $G_n$ an $(n, d_n, \alpha_n)$ graph for each $n \in N$, then
$$
i(G_n) \leq k(c_{\rm ind},c_{\rm deg})^{n\left(1+o(1)\right)}
$$
where $o(1)\rightarrow 0$ as $n, d \rightarrow \infty$.
\item[LB3] There is 
\begin{itemize}
\item an infinite set $N \subseteq {\mathbb N}$, 
\item a sequence of triples $\{(n,d_n,\alpha_n):n \in N\}$ satisfying {\bf Q1} and {\bf Q2}, and 
\item a sequence of graphs $(G_n)_{n \in N}$ with $G_n$ an $(n, d_n, \alpha_n)$ graph for each $n \in N$ 
\end{itemize}
such that
$$
i(G_n) \geq k(c_{\rm ind},c_{\rm deg})^{n\left(1+o(1)\right)}
$$
where $o(1)\rightarrow 0$ as $n, d \rightarrow \infty$.
\end{description}
\end{theorem}

As we observed after the statement of Theorem \ref{thm-ind-sets-alpha=cn-d<=n/2}, if $c_{\rm ind}=0$ then $i(G)=2^{o(n)}$ for all valid choices of $c_{\rm deg}$. Note also that $k(c_{\rm ind},1/2)=k(c_{\rm ind})$ for all $0 < c_{\rm ind} \leq 1/2$, so in establishing Theorem \ref{thm-ind-sets-alpha=cn-d=an-a>=1/2} we may assume $1/2 < c_{\rm deg} \leq 1$. Finally, note that if $c_{\rm deg}=1$ then necessarily $c_{\rm ind}=0$ so again we have a best-possible upper bound on $i(G)$ (up to $o(n)$ in the exponent).

\medskip

Kirsch and Radcliffe \cite{KirschRadcliffe} have considered the problem of maximizing the number of cliques (in fact, the number of cliques of each given size) in graphs with bounded maximum degree and with bounded clique number. Taking graph complements, this is the same as bounding the number of independent sets in a graph, in the presence of a lower bound on degrees, and an upper bound on the size of the maximum independent set. One of their main results (\cite[Theorem 3.6]{KirschRadcliffe}) gives asymptotically matching upper and lower bounds on the maximum number of independent sets (of each fixed size) in such graphs. A key difference between our work here, and the work of Kirsch and Radcliffe, is that Kirsch and Radcliffe require the size of the largest independent set to be absolutely bounded in order to obtain asymptotically tight results, whereas here we are interested mostly in the case where the size of the largest independent set grows with the number of vertices.   

\section{Summary of the paper} \label{sec-summary}

The first part ({\bf UB1}) of Theorem \ref{thm-ind-sets-alpha=cn-d<=n/2}, namely an upper bound on the number of independent sets admitted by an $(n,d,\alpha)$-graph with $\alpha \sim c_{\rm ind}n$ ($0 < c_{\rm ind} \leq 1/2$) will be derived from the following more general result.
\begin{theorem} \label{thm-generalizing-Sapozhenko}
There is an absolute constant $c > 0$ such that if $G$ is an $(n,d,\alpha)$-graph with $\alpha \leq n/2$ then
$$
i(G) \leq i(Z(\lf n/2 \rf,\alpha))\exp_2\left\{c\sqrt{(\log d)/d}\right\}.
$$
\end{theorem}
The proof of Theorem \ref{thm-generalizing-Sapozhenko}, which  is given in Section \ref{subsec-UB1}, is based on Sapozhenko's container argument for \eqref{eq-Sap-alpha-ub}. In that section we also present an analysis of the behavior of $i(Z(\lf n/2 \rf,\alpha))$, from which {\bf UB1} follows.   

The first part ({\bf UB3}) of Theorem \ref{thm-ind-sets-alpha=cn-d=an-a>=1/2}, namely an upper bound on the number of independent sets admitted by an $(n,d,\alpha)$-graph with $d \sim c_{\rm deg}n$ ($1/2 \leq c_{\rm deg} < 1$) and $\alpha \sim c_{\rm ind}n$ ($0 < c_{\rm ind} \leq 1-c_{\rm deg}$) will also be derived from a more general result, this one significantly simpler than Theorem \ref{thm-generalizing-Sapozhenko}, but in the same spirit.
\begin{theorem} \label{thm-easy-container}
If $G$ is an $(n,d,\alpha)$-graph with $d \geq n/2$ and $\alpha \leq n-d$, then
$$
i(G) \leq n i(Z(n-d,\alpha)).
$$
\end{theorem}
The proof of Theorem \ref{thm-easy-container} is described in Section \ref{subsec-UB3}. In that section we also outline an analysis of the behavior of $i(Z(n-d,\alpha))$, from which {\bf UB3} follows.  

We now discuss the lower bounds {\bf LB1} and {\bf LB3}. Recall that for $G$ an $(n,d,\alpha)$-graph with $d \lesssim n/2$  and $\alpha \sim c_{\rm ind}n$ $(0 < c_{\rm ind} \leq 1/2)$ we have
$$
i(G) \leq i(Z(\lf n/2 \rf,\alpha))2^{o(n)}.
$$
(as $d \rightarrow \infty$) and that when $d \sim c_{\rm deg}n$ ($1/2 \leq c_{\rm deg} \leq 1$) and $\alpha \sim c_{\rm ind}n$ ($0 < c_{\rm ind} \leq 1-c_{\rm deg}$) we have
$$
i(G) \leq i(Z(n-d,\alpha))2^{o(n)},
$$
In light of these two upper bounds, it would seem that the most straightforward and satisfying approach to {\bf LB1} and {\bf LB3} would be to establish that for every triple $(n,d,\alpha)$ for which there exists an $(n,d,\alpha)$-graph,  there exists an $(n,d,\alpha)$-graph that has $Z(\lf n/2 \rf, \alpha)$ as an induced subgraph (if $d \leq n/2$) or has $Z(n-d, \alpha)$ as an induced subgraph (if $d \geq n/2$), and that has maximum independent set size at most $\alpha$. 

We are not at the moment able to achieve this exact goal, but we can come close, finding a collection of constructions of $(n,d,\alpha)$-graphs that have appropriate Zykov graphs as induced subgraphs. The constructions when $d \leq n/2$ are detailed in Section \ref{sec-constructions-d<=n/2}, while the constructions when $d \geq n/2$ appear in Section \ref{sec-constructions-d>=n/2}.    We conclude with some discussion and open problems in Section \ref{sec-discussion}.

\section{Upper Bounds on independent set counts} \label{sec-UB1,3}

\subsection{Proof of {\bf UB1}} \label{subsec-UB1}

\subsubsection{Proof of Theorem \ref{thm-generalizing-Sapozhenko}}

We begin with the proof of Theorem \ref{thm-generalizing-Sapozhenko}. Our proof follows Sapozhenko's proof of \eqref{eq-Sap-alpha-ub}, replacing at a key point an appeal to a theorem of Alekseev with an appeal to Zykov's generalization of Tur\'an's Theorem.    

The heart of Sapozhenko's proof of \eqref{eq-Sap-alpha-ub} is the following container lemma (\cite{Sapozhenko2}, see also \cite{Galvin} for an exposition).  
\begin{lemma} \label{lem-Sapozhenko}
Let $G$ be an $(n,d)$-graph, and let $0 < \varphi < d$ be an integer. There is a family ${\mathcal D}$ of subsets of the vertex set of $G$, with the following properties:
\begin{enumerate}
\item $|{\mathcal D}| \leq \sum_{i \leq n/\varphi} \binom{n}{i}$.
\item Each $D \in {\mathcal D}$ satisfies
$$
|D| \leq \frac{nd}{2d-\varphi}.
$$
\item For each independent set $I$ in $G$, there is $D \in {\mathcal D}$ with $I \subseteq D$. 
\end{enumerate}
\end{lemma}
We may think of ${\mathcal D}$ as a set of containers that between them include all the independent sets of $G$; upper bounds on the number and size of the containers combine to give an upper bound on the number of independent sets in $G$. 

Another ingredient in the proof is a result of Zykov \cite{Zykov} (Theorem \ref{thm-Zykov} below, alluded to in the introduction) that generalizes Tur\'an's Theorem. For integers $\alpha$ and $n$ with $1 \leq \alpha \leq n$, the {\it Zykov graph} $Z(n,\alpha)$ is the disjoint union of $\alpha$ cliques, as near equal in size as possible. Note that the graph complement of $Z(n,\alpha)$ is the $n$-vertex Tur\'an graph with $\alpha$ classes --- the complete multipartite graph on $n$ vertices with $\alpha$ partite classes, as near equal in size as possible. If $\alpha$ divides $n$ then $Z(n,\alpha)$ is the disjoint union of $\alpha$ cliques of size $n/\alpha$. Otherwise, $Z(n,\alpha)$ consists of $\alpha\lc n/\alpha\rc -n$ cliques of size $\lf n/ \alpha \rf$ and $n-\alpha\lf n/\alpha\rf$ 
cliques of size $\lc n/ \alpha\rc$. 

For a positive integer $t$ and a graph $G$ denote by $i_t(G)$ the number of independent sets in $G$ of size $t$.

\begin{theorem}[\cite{Zykov}] \label{thm-Zykov}
Let $G$ be a graph on $n$ vertices with maximum independent set size at most $\alpha$. We have
$$
i(G) \leq i(Z(n,\alpha))
$$
and more generally for all $1 \leq t \leq n$ we have
$$
i_t(G) \leq i_t(Z(n,\alpha)).
$$
\end{theorem}
Setting $t=2$ and taking graphs complements, Theorem \ref{thm-Zykov} reduces to Tur\'an's Theorem. 

Now let $G$ be an $(n,d,\alpha)$-graph. By Lemma \ref{lem-Sapozhenko} there is a set ${\mathcal D}$ of size at most $\sum_{i \leq n/\varphi} \binom{n}{i}$ such that each independent set in $G$ is contained in some $D \in {\mathcal D}$, and so is an independent set in the subgraph $G[D]$ of $G$ induced by $D$. Since any induced subgraph of $G$ inherits the property from $G$ that its maximum independent set size is at most $\alpha$, we have from Theorem \ref{thm-Zykov} that for any such $D$
\begin{equation} \label{using-Zykov}
i(G[D]) \leq i(Z(|D|,\alpha)) \leq i(Z(\lf nd/(2d-\varphi) \rf, \alpha)).
\end{equation}
The second inequality in \eqref{using-Zykov} above follows from that fact that $|D| \leq \lf nd/(2d-\varphi) \rf$ for all $D \in {\mathcal D}$, and from a simple monotonicity observation regarding $i(Z(\cdot,\alpha))$. In \eqref{using-Zykov} we use only the first inequality from Lemma \ref{lem-add-one-in-Zykov} below; the second inequality will be useful later.
\begin{lemma} \label{lem-add-one-in-Zykov}
For any $n$ and $\alpha$ with $1 \leq \alpha \leq n$, and any positive integer $k$, 
$$
i(Z(n,\alpha)) \leq i(Z(n+k,\alpha))\leq (3/2)^k i(Z(n,\alpha)).
$$
\end{lemma}

\begin{proof}
The structural difference between $Z(n+1,\alpha)$ and $Z(n,\alpha)$ is that one of the smaller cliques (one of the cliques of size $\lf n/\alpha \rf$) in $Z(n,\alpha)$ becomes a clique of size $\lf n/\alpha \rf+1$ in $Z(n+1,\alpha)$. This implies that
\begin{equation} \label{eq-Z-change}
i(Z(n+1,\alpha))=\frac{(2+\lf n/\alpha \rf)}{(1+\lf n/\alpha \rf)}i(Z(n,\alpha)).
\end{equation}

The right-hand inequality of the lemma now follows by induction from \eqref{eq-Z-change} and from the fact that $(2+x)/(1+x) \leq 3/2$ for $x \geq 1$. The left-hand inequality of the lemma follows by induction from \eqref{eq-Z-change} and from the fact that $(2+x)/(1+x) \geq 1$ for $x > -1$. 
\end{proof}

We have obtained
\begin{equation} \label{inq-Sap-Zyz-ub}
i(G) \leq \left(\sum_{i \leq n/\varphi} \binom{n}{i}\right) i(Z(\lf nd/(2d-\varphi) \rf,\alpha)).
\end{equation}

In the sequel we will select a value for $\varphi$ that is both $\omega(1)$ and $o(d)$ as $d \rightarrow \infty$. With this choice of $\varphi$ we have
\begin{eqnarray}
\sum_{i \leq n/\varphi} \binom{n}{i} & \leq & \sum_{i \leq n/\varphi} \frac{n^i}{i!} \nonumber \\
& = & \sum_{i \leq n/\varphi} \frac{(n/\varphi)^i}{i!} \varphi^i  \nonumber  \\
& \leq & e^{n/\varphi} \varphi^{n/\varphi}  \nonumber  \\
& \leq & 2^{O\left(\frac{n\log \varphi}{\varphi}\right)} \label{est1}
\end{eqnarray}
and
\begin{eqnarray}
\lf \frac{nd}{2d-\varphi} \rf & \leq & \frac{n}{2} + O\left(\frac{n\varphi}{d}\right). \label{est2}
\end{eqnarray}
From \eqref{est2} and the left-hand inequality in Lemma \ref{lem-add-one-in-Zykov} we have
\begin{equation} \label{est3}
i(Z(\lf nd/(2d-\varphi) \rf,\alpha)) \leq i(Z(\lf n/2 + O(n\varphi/d) \rf,\alpha)).
\end{equation}
Applying the right-hand inequality of Lemma \ref{lem-add-one-in-Zykov} to \eqref{est3} we obtain
\begin{equation} \label{est4}
i(Z(\lf nd/(2d-\varphi) \rf,\alpha)) \leq i(Z(\lf n/2 \rf,\alpha))2^{O\left(\frac{n\varphi}{d}\right)}.
\end{equation}
Inserting \eqref{est1} and \eqref{est4} into \eqref{inq-Sap-Zyz-ub} we get
$$
i(G) \leq i(Z(\lf n/2 \rf,\alpha))2^{O\left(\frac{n\log \varphi}{\varphi} + \frac{n\varphi}{d}\right)}.
$$
Setting $\varphi = \lf \sqrt{d \log d} \rf$ we obtain Theorem \ref{thm-generalizing-Sapozhenko}.

\subsubsection{Deriving {\bf UB1} from Theorem \ref{thm-generalizing-Sapozhenko}} \label{limit-argument}

We now study the behavior of $i(Z(\lf n/2 \rf, \alpha))$, in order to derive {\bf UB1} from Theorem \ref{thm-generalizing-Sapozhenko}. Our goal is to establish the following lemma.
\begin{lemma} \label{lem-z-asy}
$$
i(Z(\lf n/2 \rf, \alpha_n)) = k(c_{\rm ind})^{n+o(n)}
$$
when $\alpha_n = c_{\rm ind}n + o(n)$ as $n \rightarrow \infty$, where recall
$$
k(c_{\rm ind}) =
    \begin{cases} 
      \left(1+\frac{1}{2c_{\rm ind}}\right)^{c_{\rm ind}} & \text{if } 1/(2c_{\rm ind}) \text{ is an integer}  \\
    \left(1+\lf\frac{1}{2c_{\rm ind}}\rf\right)^{c_{\rm ind}\lc\frac{1}{2c_{\rm ind}}\rc-1/2} \left(1+\lc\frac{1}{2c_{\rm ind}}\rc\right)^{1/2-c_{\rm ind}\lf\frac{1}{2c_{\rm ind}}\rf} & \text{otherwise.}
   \end{cases}
$$
\end{lemma}

\begin{proof}
Recall that if $\alpha$ divides $N$ then $Z(N,\alpha)$ is the disjoint union of $\alpha$ cliques of size $N/\alpha$. Otherwise, $Z(N,\alpha)$ consists of $\alpha\lc N/\alpha\rc -N$ cliques of size $\lf N/ \alpha \rf$ and $N-\alpha\lf N/\alpha\rf$ 
cliques of size $\lc N/ \alpha\rc$. It follows that for arbitrary $n$ and $\alpha$ we have
$$
i(Z(\lf n/2 \rf, \alpha)) = 
\begin{cases} 
      \left(1+ \frac{\lf n/2 \rf}{\alpha}\right)^\alpha & \text{if } \alpha \mid \lf n/2 \rf\\
      \left(1+ \lf \frac{\lf n/2 \rf}{\alpha} \rf\right)^{\alpha\lc \frac{\lf n/2 \rf}{\alpha} \rc -\lf \frac{n}{2} \rf} \left(1+ \lc \frac{\lf n/2 \rf}{\alpha} \rc\right)^{\lf \frac{n}{2} \rf - \alpha\lf \frac{\lf n/2 \rf}{\alpha}\rf} & \text{otherwise.}
   \end{cases}
$$

A useful observation in all of what follows is that $\alpha_n = c_{\rm ind}n+o(n)$ (together with $\lf n/2 \rf =n/2 + o(1)$) implies that
\begin{equation} \label{dealing-with-int-part}
\frac{\lf n/2 \rf}{\alpha_n} =\frac{1}{2c_{\rm ind}} + o(1).
\end{equation}

We consider first the case when $1/(2c_{\rm ind})$ is an integer. By \eqref{dealing-with-int-part} we may assume that $n$ is large enough that 
$$
\frac{1}{2c_{\rm ind}} - \frac{1}{3} \leq \frac{\lf n/2 \rf}{\alpha_n} \leq \frac{1}{2c_{\rm ind}} + \frac{1}{3} 
$$
($1/3$ is quite arbitrary here --- anything strictly less than $1$ will do).

We treat three subcases --- first $\lf n/2 \rf/\alpha_n=1/(2c_{\rm ind})$, then $\lf n/2 \rf/\alpha_n > 1/(2c_{\rm ind})$, and finally $\lf n/2 \rf/\alpha_n < 1/(2c_{\rm ind})$.

\begin{itemize}
\item If $\lf n/2 \rf/\alpha_n=1/(2c_{\rm ind})$ (so $\alpha_n$ divides $\lf n/2 \rf$) we have immediately
$$
i(Z(\lf n/2 \rf, \alpha_n)) = \left(1+\frac{1}{2c_{\rm ind}}\right)^{c_{\rm ind}n+o(n)} = k(c_{\rm ind})^{n+o(n)}.
$$
\item
If $\lf n/2 \rf/\alpha_n > 1/(2c_{\rm ind})$ then we have
$$
\lf \frac{\lf n/2 \rf}{\alpha_n} \rf = \frac{1}{2c_{\rm ind}}~~~\mbox{and}~~~ \lc \frac{\lf n/2 \rf}{\alpha_n} \rc = 1+\frac{1}{2c_{\rm ind}}.
$$
This, together with $\lf n/2 \rf = n/2+o(1)$ says that
\begin{eqnarray}
i(Z(\lf n/2 \rf, \alpha_n)) & = & \left(1+\frac{1}{2c_{\rm ind}}\right)^{(c_{\rm ind}n + o(n))(1/(2c_{\rm ind})+1)-(n/2)+o(1)} \nonumber\\
& & ~~~~~~~~~~~~~~~~~~\times \left(2+\frac{1}{2c_{\rm ind}}\right)^{n/2+o(1)-(c_{\rm ind}n+o(n))/(2c_{\rm ind})} \label{col1} \\
& = & \left(1+\frac{1}{2c_{\rm ind}}\right)^{c_{\rm ind}n+o(n)} \nonumber \\
& = & k(c_{\rm ind})^{n+o(n)}, \nonumber
\end{eqnarray}
the main point being that the second term on the right-hand side of \eqref{col1} collapses to $2^{o(n)}$. 
\item
If $\lf n/2 \rf/\alpha_n < 1/(2c_{\rm ind})$ then we have
$$
\lf \frac{\lf n/2 \rf}{\alpha_n} \rf = \frac{1}{2c_{\rm ind}}-1~\mbox{and}~ \lc \frac{\lf n/2 \rf}{\alpha_n} \rc = \frac{1}{2c_{\rm ind}},
$$
and again we get $i(Z(\lf n/2 \rf, \alpha_n))=k(c_{\rm ind})^{n+o(n)}$ (this time because the first term in the analog of \eqref{col1} collapses).
\end{itemize}
So in this case ($1/(2c_{\rm ind})$ an integer) we always have $i(Z(\lf n/2 \rf, \alpha_n))=k(c_{\rm ind})^{n+o(n)}$, as claimed.   

Next we consider the case where $1/(2c_{\rm ind})$ is not an integer. By \eqref{dealing-with-int-part} we may assume that $n$ is large enough that $\lf n/2 \rf/\alpha_n$ and $1/(2c_{\rm ind})$ share the same floor and the same ceiling, and that moreover $\alpha_n$ does not divide $\lf n/2 \rf$. We have immediately that
\begin{eqnarray*}
i(Z(\lf n/2 \rf, \alpha_n)) & = & \left(1+\lf \frac{1}{2c_{\rm ind}} \rf\right)^{(c_{\rm ind}n+o(n))\lc \frac{1}{2c_{\rm ind}} \rc - n/2 + o(1)} \\
& & ~~~~~~~~~~~~~~~~~~\times \left(1+\lc \frac{1}{2c_{\rm ind}} \rc\right)^{n/2+o(1)-(c_{\rm ind}n+o(n))\lf \frac{1}{2c_{\rm ind}} \rf} \\
& = & k(c_{\rm ind})^{n+o(n)}.
\end{eqnarray*}
\end{proof}
This completes the derivation of {\bf UB1} from Theorem \ref{thm-generalizing-Sapozhenko}

\subsection{Proof of {\bf UB3}} \label{subsec-UB3}

\subsubsection{Proof of Theorem \ref{thm-easy-container}}

As with Theorem \ref{thm-generalizing-Sapozhenko} this involves a container argument, although a far simpler one than used previously.

Let $G$ be an $(n,d,\alpha)$-graph with $d \geq n/2$ and $1 \leq \alpha \leq n-d$, let $I$ be an independent set in $G$, and let $v$ be a vertex in $I$. As $N(v)$ is disjoint from $I$ it follows that $I$ is a subset of $V(G)\setminus N(v)$, a set of size $n-d$. So $\{V(G)\setminus N(v): v \in V(G)\}$ forms a set of $n$ containers that between them include all independent sets of $G$ as subsets. Noting that any subgraph of $G$ inherits from $G$ that its maximum independent set size is at most $\alpha$, we may apply Theorem \ref{thm-Zykov} to conclude that each container admits at most $i(Z(n-d,\alpha))$ indepenedent sets, and hence $i(G)\leq n i(Z(n-d,\alpha))$. 

\subsubsection{Deriving {\bf UB3} from Theorem \ref{thm-easy-container}}

We now study the behavior of $Z(n-d,\alpha)$, in order to derive {\bf UB3} from Theorem \ref{thm-easy-container}. Noting that for any constant $k(c_{\rm ind},c_{\rm deg})$ we have $n=k(c_{\rm ind},c_{\rm deg})^{o(n)}$, our goal is to show that
\begin{equation} \label{to-show-zykov-large-d}
i(Z(n-d_n,\alpha_n)) = k(c_{\rm ind},c_{\rm deg})^{n+o(n)}
\end{equation}
when $d_n=c_{\rm deg}n+o(n)$ ($1/2 \leq c_{\rm deg} \leq 1$) and $\alpha_n = c_{\rm ind}n+o(n)$ ($0 < c_{\rm ind} \leq 1-c_{\rm deg}$), where recall
$$
k(c_{\rm ind},c_{\rm deg}) = \begin{cases}
    (1+\frac{1-c_{\rm deg}}{c_{\rm ind}})^{c_{\rm ind}}  & \frac{1-c_{\rm deg}}{c_{\rm ind}} \in {\mathbb N}\\
    \left(1+\lf \frac{1-c_{\rm deg}}{c_{\rm ind}} \rf \right)^{c_{\rm ind}\lc \frac{1-c_{\rm deg}}{c_{\rm ind}} \rc-1+c_{\rm deg}}  \left(1+\lc \frac{1-c_{\rm deg}}{c_{\rm ind}} \rc \right) ^{1-c_{\rm deg}-\lf \frac{1-c_{\rm deg}}{c_{\rm ind}}\rf} & \text{otherwise.}
\end{cases}
$$

As we have previously observed, $Z(n-d,\alpha)$ consists of
\begin{itemize}
\item the disjoint union of $\alpha$ cliques of size $(n-d)/\alpha$, if $\alpha$ divides $n-d$, and
\item $\alpha \lc (n-d)/\alpha \rc -n+d$ cliques of size $\lf (n-d)/\alpha \rf$ and $n-d-\alpha \lf (n-d)/\alpha \rf$ cliques of size $\lc (n-d)/\alpha \rc$, otherwise.
\end{itemize}. It follows that
$$
i(Z(n-d,\alpha))=
\begin{cases} 
      \left(1+\frac{n-d}{\alpha} \right)^\alpha & \text{if } \alpha \text{ divides } n-d\\
      \left(1+\lf \frac{n-d}{\alpha} \rf \right)^{\alpha\lc \frac{n-d}{\alpha} \rc -n+d}\left(1+\lc \frac{n-d}{\alpha} \rc \right)^{n-d-\alpha\lf \frac{n-d}{\alpha} \rf} & \text{otherwise.}
   \end{cases}
$$

Noting that $(n-d_n)/\alpha_n=(1-c_{\rm deg})/c_{\rm ind}+o(1)$, we can use an argument almost identical to the one used in Section \ref{limit-argument} (derivation of {\bf UB1}) to show that indeed \eqref{to-show-zykov-large-d} holds here.

\section{Constructing regular graphs with Zykov graphs as subgraphs --- the case $d \leq n/2$} \label{sec-constructions-d<=n/2}

In this section and in Section \ref{sec-constructions-d>=n/2} we establish {\bf LB1} and {\bf LB3}, and also show that the upper bounds {\bf UB1} and {\bf UB3} are essentially tight when $\alpha$ (and, if necessary, $d$) scale linearly with $n$. As discussed in Section \ref{sec-summary}, ideally here we would establish that for every triple $(n,d,\alpha)$ for which there exists an $(n,d,\alpha)$-graph, there exists one that has $Z(\lf n/2 \rf, \alpha)$ as an induced subgraph (if $d \leq n/2$) or has $Z(n-d, \alpha)$ as an induced subgraph (if $d \geq n/2$), and has maximum independent set size at most $\alpha$. We are not quite able to achieve that goal, but we will come close enough to be able to easily establish {\bf LB1} and {\bf LB3}. 

This section (Section \ref{sec-constructions-d<=n/2}) is focused on the case $n \leq d/2$ ({\bf LB1}), while Section \ref{sec-constructions-d>=n/2} is devoted to the case $d \geq n/2$ ({\bf LB3}). We begin in Section \ref{subsec-constructions-d<=n/2-n-even} by considering the subcase of even $n$. In Section \ref{subsec-constructions-d<=n/2-n-odd} we describe the necessary modifications to deal with odd $n$. In Section \ref{subsec-proof-of-LB1} we use the constructions of Sections \ref{subsec-constructions-d<=n/2-n-even} and \ref{subsec-constructions-d<=n/2-n-odd} to establish {\bf LB1}. 

\subsection{Constructions for $n$ even} \label{subsec-constructions-d<=n/2-n-even}

We start all our constructions with two disjoint copies of $Z(n/2,\alpha)$, and we will not add any further edges inside either copy; this ensures that the final graphs have $Z(n/2,\alpha)$ as an induced subgraph. We will then pair up the cliques in the two copies of $Z(n/2,d)$, and add complete bipartite graphs between the two cliques in each pair; this ensures that the final graphs will have independence number at most $\alpha$. Since this base graph is common to all our constructions, we establish some conventions and define it precisely in Section \ref{subsec-base-graph-1}.

If necessary, we then add a few more edges to $G_1(n,\alpha)$ to make sure that the graph thus far constructed is regular.

Next we identify a high-degree regular bipartite subgraph among the edges between the two copies of $Z(n/2,\alpha)$, using only edges that have not already been used in $G_1(n,\alpha)$, or used to make the graph regular. We repeatedly add perfect matchings from this graph of potential edges to complete the construction of an appropriate $(n,d,\alpha)$-graph $G(n,d,\alpha)$. 

We will have to consider various cases (Sections \ref{subsec-construction-d<=n/s-n-even-1}, \ref{subsec-construction-d<=n/s-n-even-2}, \ref{subsec-construction-d<=n/s-n-even-3}, \ref{subsec-construction-d<=n/s-n-even-4} and \ref{subsec-construction-d<=n/s-n-even-5}) depending on the precise relationship between $n$ and $\alpha$.     

\subsubsection{The base graph $G_1(n,\alpha)$} \label{subsec-base-graph-1}

First, note that $Z(n/2, \alpha)$ is the disjoint union of $n_a$ cliques of order $a$ and $n_b$ cliques of order $b$, with $|a-b| = 1$ unless $\alpha$ divides $n/2$, in which case $a=b$ and it is convenient to take $n_a=\alpha$ and $n_b=0$. Otherwise we assume (without loss of generality) that $n_a \geq n_b$. We construct a graph $G_1(n,\alpha)$ on vertex set $X\cup Y$ where $X$ and $Y$ are disjoint sets each of size $n/2$ as follows. On $X$ we place a copy of $Z(n/2,\alpha)$ with cliques $X_1,\ldots,X_{n_a}$ of order $a$ and $X_{n_a+1},\ldots,X_\alpha$ of order $b$. On $Y$ we place a copy of $Z(n/2,\alpha)$ with cliques $Y_1,\ldots,Y_{n_b}$ of order $b$ and $Y_{n_b+1},\ldots,Y_\alpha$ of order $a$. Then for each $i=1, \ldots, \alpha$ we add a complete bipartite graph with partition classes $X_i$ and $Y_i$. Observe that $G_1(n,\alpha)$, being the union of $\alpha$ cliques, has largest independent set size $\alpha$. Since all our constructions will add edges onto the base graph $G_1(n,\alpha)$, and then only between the two copies of $Z(n/2,\alpha)$, this guarantees both that our constructions will have independence number $\alpha$ and that they will have $Z(n/2,\alpha)$ as an induced subgraph.

\subsubsection{When $\alpha$ divides $n/2$} \label{subsec-construction-d<=n/s-n-even-1}

\begin{theorem} \label{thm-alpha-divides-n/2}
Fix $n, d$ and $\alpha$ with $n \geq 2$ and even, $1 \leq \alpha \leq n/2$, $n/2$ divisible by $\alpha$, and $(n/\alpha)-1 \leq d \leq n/2+n/(2\alpha)-1$. There is an $(n,d,\alpha)$-graph $G(n,d,\alpha)$ that has $Z(n/2,\alpha)$ as an induced subgraph.  
\end{theorem}

\begin{proof}
We begin with $G_1(n,\alpha)$ as defined in Section \ref{subsec-base-graph-1}. Since $G_1(n,\alpha)$ is already regular (all vertices have degree $(n/\alpha)-1$) in this case we do not need to add any edges to  $G_1(n,\alpha)$ to make it regular, so here we set $G_2(n,\alpha):=G_1(n,\alpha)$. (In some later cases $G_2(n,\alpha)$ will be a proper supergraph of $G_1(n,\alpha)$.) 

Now let $G_3(n,\alpha)$ be the graph on vertex set $X \cup Y$ whose edges are all those pairs $\{x,y\}$ with $x \in X$ and $y \in Y$ such that $\{x,y\}$ is not an edge of $G_2(n,\alpha)$. Observe that $G_3(n,\alpha)$ is an $((n/2)-(n/(2\alpha)))$-regular bipartite graph. It follows (via a standard application of Hall's marriage theorem) that for every $d'$ satisfying $0 \leq d' \leq (n/2)-(n/(2\alpha))$, $G_3(n,\alpha)$ has a $d'$-regular subgraph. Select one with $d'=d-((n/\alpha)-1)$, and call it $G_4(n,d,\alpha)$. Then $G(n,d,\alpha)=G_2(n,\alpha)\cup G_4(n,d,\alpha)$ is an $(n,d,\alpha)$-graph that has $Z(n/2,\alpha)$ as an induced subgraph.
\end{proof}

\subsubsection{When $\alpha$ does not divide $n/2$, and $n_a=n_b$} \label{subsec-construction-d<=n/s-n-even-2}

Here and in subsequent parts of Section \ref{subsec-constructions-d<=n/2-n-even}, see Section \ref{subsec-base-graph-1} for the definitions of $a, b, n_a$ and $n_b$.

\begin{theorem} \label{thm-alpha-not-divide-n/2-n_a=n_b}
Fix $n, d$ and $\alpha$ with $n \geq 2$ and even, $1 \leq \alpha \leq n/2$, $n/2$ not divisible by $\alpha$, $n_a=n_b$, and $a+b-1\leq d \leq n/2+a-1$. There is an $(n,d,\alpha)$-graph $G(n,d,\alpha)$ that has $Z(n/2,\alpha)$ as an induced subgraph.  
\end{theorem}

\begin{proof}
Since $n_a=n_b$ we may assume without loss of generality that $b=a+1$. As in the proof of Theorem \ref{thm-alpha-divides-n/2} we begin with $G_1(n,\alpha)$, which in this case is $(a+b-1)$-regular, so again we set $G_2(n,\alpha):=G_1(n,\alpha)$.

While $G_2(n,\alpha)$ is a regular graph, if we defined $G_3(n,\alpha)$ as in Theorem \ref{thm-alpha-divides-n/2} the resulting graph would not be regular. This is because in $G_2(n,\alpha)$, vertices in $\cup_{1 \leq i \leq n_a} X_i$ have $(n/2)-b$ non-neighbors among $\cup_{1 \leq j \leq \alpha} Y_i$, while those in $\cup_{n_a+1 \leq i \leq \alpha} X_i$ have $(n/2)-a$ non-neighbors among $\cup_{1 \leq j \leq \alpha} Y_i$. Similarly vertices in $\cup_{1 \leq i \leq n_b} Y_i$ have $(n/2)-a$ non-neighbors among $\cup_{1 \leq j \leq \alpha} X_i$, while those in $\cup_{n_b+1 \leq i \leq \alpha} Y_i$ have $(n/2)-b$ non-neighbors among $\cup_{1 \leq j \leq \alpha} X_i$. To put this observation in other words: unlike the situation in Theorem \ref{thm-alpha-divides-n/2}, the collection of thus-far-unused edges between the two copies of $Z(n/2,\alpha)$ in $G_2(n,\alpha)$ does not form a regular bipartite graph. 

With the goal of identifying a high-degree regular bipartite spanning subgraph on $X \cup Y$ that has partition classes $X$ and $Y$, and that does not use any edges from $G_2(n,\alpha)$, let $F(n,\alpha)$ be the graph on $X\cup Y$ that consists of a perfect matching between $X_{i+n_a}$ and $Y_i$ for $i=1,\ldots,n_b$, and has no other edges. We think of $F(n,\alpha)$ as a collection of forbidden edges, that cannot appear in the final construction of $G(n,d,\alpha)$. 

Let us now define $G_3(n,\alpha)$ to be the graph on $X\cup Y$ whose edges are all the pairs $\{x,y\}$ with $x\in X$ and $y \in Y$ such that $\{x,y\}$ is neither an edge of $G_2(n,\alpha)$ nor an edge of $F(n,\alpha)$. 

Observe that $G_3(n,\alpha)$ is an $((n/2)-a-1)$-regular bipartite graph, so it follows that for every $d'$ satisfying $0\leq d'\leq (n/2)-a-1$, $G_2(n,\alpha)$ has a $d'$-regular subgraph. Select one with $d'=d-2a$, and call it $G_4(n,d,\alpha)$. Then $G(n,d,\alpha)=G_2(n,\alpha)\cup G_4(n,d,\alpha)$ is an $(n,d,\alpha)$-graph that has $Z(n/2,\alpha)$ as an induced subgraph.
\end{proof}

Figure \ref{fig4_2} shows some of the intermediate steps of this construction. In this and other figures in Section \ref{subsec-constructions-d<=n/2-n-even} red lines represent complete bipartite graphs between sets of vertices and dotted black lines represent forbidden edges (the edges of $F(n,\alpha)$; so in this case the dotted black lines represent perfect matchings between sets of vertices). Figure \ref{fig4_2} presents the specific situation when $n_a=n_b=3$, but the construction for general $n_a$ and $n_b$ should be clear. 

\begin{figure}[ht!] 
\centering
\begin{tikzpicture}
\node[] at (0,0) {$K_b$};
\node[] at (0,1) {$K_b$};
\node[] at (0,2) {$K_b$};
\node[] at (0,3) {$K_a$};
\node[] at (0,4) {$K_a$};
\node[] at (0,5) {$K_a$};

\node[] at (2,0) {$K_a$};
\node[] at (2,1) {$K_a$};
\node[] at (2,2) {$K_a$};
\node[] at (2,3) {$K_b$};
\node[] at (2,4) {$K_b$};
\node[] at (2,5) {$K_b$};

\draw[red,ultra thick] (0.3,0)--(1.7,0);
\draw[red,ultra thick] (0.3,1)--(1.7,1);
\draw[red,ultra thick] (0.3,2)--(1.7,2);
\draw[red,ultra thick] (0.3,3)--(1.7,3);
\draw[red,ultra thick] (0.3,4)--(1.7,4);
\draw[red,ultra thick] (0.3,5)--(1.7,5);

\node[] at (1,-1) {$G_2(n,\alpha)$};

\node[] at (10,0) {$K_b$};
\node[] at (10,1) {$K_b$};
\node[] at (10,2) {$K_b$};
\node[] at (10,3) {$K_a$};
\node[] at (10,4) {$K_a$};
\node[] at (10,5) {$K_a$};

\node[] at (12,0) {$K_a$};
\node[] at (12,1) {$K_a$};
\node[] at (12,2) {$K_a$};
\node[] at (12,3) {$K_b$};
\node[] at (12,4) {$K_b$};
\node[] at (12,5) {$K_b$};

\draw[red,ultra thick] (10.3,0)--(11.7,0);
\draw[red,ultra thick] (10.3,1)--(11.7,1);
\draw[red,ultra thick] (10.3,2)--(11.7,2);
\draw[red,ultra thick] (10.3,3)--(11.7,3);
\draw[red,ultra thick] (10.3,4)--(11.7,4);
\draw[red,ultra thick] (10.3,5)--(11.7,5);

\draw[dashed,ultra thick] (10.3,0)--(11.7,3);
\draw[dashed,ultra thick] (10.3,1)--(11.7,4);
\draw[dashed,ultra thick] (10.3,2)--(11.7,5);

\node[] at (11,-1) {$G_2(n,\alpha)\cup F(n,\alpha)$};

\end{tikzpicture}
\caption{Illustration of some aspects of the construction when $n_a=n_b$, $b=a+1$, with $n$ even.}
\label{fig4_2} 
\end{figure}
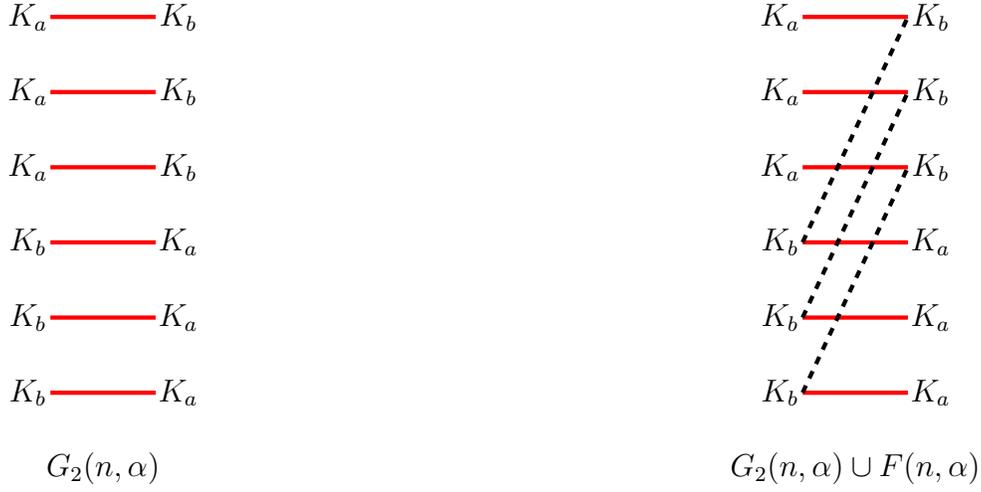 

\subsubsection{When $n_a>n_b$ and $b=a-1$} \label{subsec-construction-d<=n/s-n-even-3}

\begin{theorem} \label{thm-alpha-not-divide-n/2-n_a>n_b-b=a-1}
Fix $n, d$ and $\alpha$ with $n \geq 2$ and even, $1 \leq \alpha \leq n/2$, $n/2$ not divisible by $\alpha$, $n_a>n_b$, $b=a-1$, and $a+b\leq d \leq n/2+a-2$. There is an $(n,d,\alpha)$-graph $G(n,d,\alpha)$ that has $Z(n/2,\alpha)$ as an induced subgraph.  
\end{theorem}

\begin{proof}
As in the previous cases, we start with $G_1(n,\alpha)$ as our base. Unlike the previous two cases, however, $G_1(n,\alpha)$ is not regular. Indeed, the degree of a vertex in either $X_i$ or $Y_i$, $1 \leq i \leq n_b$, is $a+b-1$, as is the degree of a vertex in either $X_{\alpha+1-i}$ or $Y_{\alpha+1-i}$, $1 \leq i \leq n_b$. However, vertices in $X_i$ or $Y_i$, $n_b+1 \leq i \leq n_a$, have degree $2a-1$, which is one larger than $a+b-1$.

We augment $G_1(n,\alpha)$ by adding a perfect matching between $X_i$ and $Y_{\alpha+1-i}$ for $i=1,\ldots, n_b$ and a perfect matching between $X_{\alpha+1-j}$ and $Y_j$ for $j=1,\ldots,n_b$. The resulting graph, which we name $G_2(n,\alpha)$, is $(2a-1)$-regular and still has independence number $\alpha$.

As in the proof of Theorem \ref{thm-alpha-not-divide-n/2-n_a=n_b}, we next identify a collection of edges between $X$ and $Y$ that do not appear in $G_2(n,\alpha)$, such that if those edges are removed from the as-yet-unused edges between $X$ and $Y$, the remaining edges form a high-degree regular graph. Specifically, define $F(n,\alpha)$ to be the graph on $X\cup Y$ which consists of a perfect matching between $X_i$ and $Y_{n_b+i}$ for $i=1,\ldots,n_a$, a perfect matching between $X_{n_a+j}$ and $Y_i$ for $i=1,\ldots,n_a$, and no other edges. 

Now let $G_3(n,\alpha)$ be the graph on $X\cup Y$ whose edges are all the pairs $\{x,y\}$ with $x\in X$ and $y \in Y$ such that $\{x,y\}$ is neither an edge of $G_2(n,\alpha)$ nor of $F(n,\alpha)$. Observe that $G_3(n,\alpha)$ is an $((n/2)-a-1)$-regular bipartite graph, so it follows that for every $d'$ satisfying $0\leq d'\leq n/2-a-1$, $G_3(n,\alpha)$ has a $d'$-regular subgraph. Select one with $d'=d-(2a-1)$, and call it $G_4(n,d,\alpha)$. Then $G(n,d,\alpha)=G_2(n,\alpha)\cup G_4(n,d,\alpha)$ is an $(n,d,\alpha)$-graph that has $Z(n/2,\alpha)$ as an induced subgraph.
\end{proof}

Figure \ref{fig4_3} shows some of the intermediate steps of this construction. As in the case of Figure \ref{fig4_2}, red lines here represent complete bipartite graphs between sets of vertices and dotted black lines represent forbidden perfect matchings. 
The figure here also includes solid black lines, which here and elsewhere in Section \ref{subsec-constructions-d<=n/2-n-even} represent perfect matchings between sets of vertices. Figure \ref{fig4_3} specifically represents the situation when $n_a=5$ and $n_b=2$, but the construction for general $n_a$ and $n_b$ should be clear.

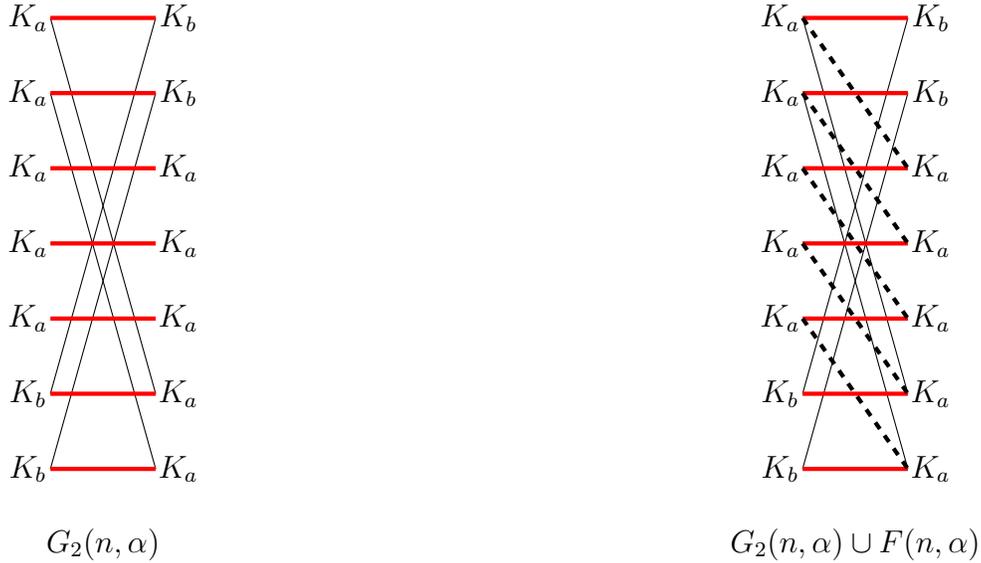
\begin{figure}[ht!] 
\centering
\begin{tikzpicture}
\node[] at (0,0) {$K_b$};
\node[] at (0,1) {$K_b$};
\node[] at (0,2) {$K_a$};
\node[] at (0,3) {$K_a$};
\node[] at (0,4) {$K_a$};
\node[] at (0,5) {$K_a$};
\node[] at (0,6) {$K_a$};

\node[] at (2,0) {$K_a$};
\node[] at (2,1) {$K_a$};
\node[] at (2,2) {$K_a$};
\node[] at (2,3) {$K_a$};
\node[] at (2,4) {$K_a$};
\node[] at (2,5) {$K_b$};
\node[] at (2,6) {$K_b$};

\draw (0.3,0)--(1.7,5);
\draw (0.3,1)--(1.7,6);
\draw (0.3,6)--(1.7,1);
\draw (0.3,5)--(1.7,0);

\draw[red,ultra thick] (0.3,0)--(1.7,0);
\draw[red,ultra thick] (0.3,1)--(1.7,1);
\draw[red,ultra thick] (0.3,2)--(1.7,2);
\draw[red,ultra thick] (0.3,3)--(1.7,3);
\draw[red,ultra thick] (0.3,4)--(1.7,4);
\draw[red,ultra thick] (0.3,5)--(1.7,5);
\draw[red,ultra thick] (0.3,6)--(1.7,6);

\node[] at (1,-1) {$G_2(n,\alpha)$};

\node[] at (10,0) {$K_b$};
\node[] at (10,1) {$K_b$};
\node[] at (10,2) {$K_a$};
\node[] at (10,3) {$K_a$};
\node[] at (10,4) {$K_a$};
\node[] at (10,5) {$K_a$};
\node[] at (10,6) {$K_a$};

\node[] at (12,0) {$K_a$};
\node[] at (12,1) {$K_a$};
\node[] at (12,2) {$K_a$};
\node[] at (12,3) {$K_a$};
\node[] at (12,4) {$K_a$};
\node[] at (12,5) {$K_b$};
\node[] at (12,6) {$K_b$};

\draw (10.3,0)--(11.7,5);
\draw (10.3,1)--(11.7,6);
\draw (10.3,6)--(11.7,1);
\draw (10.3,5)--(11.7,0);

\draw[red,ultra thick] (10.3,0)--(11.7,0);
\draw[red,ultra thick] (10.3,1)--(11.7,1);
\draw[red,ultra thick] (10.3,2)--(11.7,2);
\draw[red,ultra thick] (10.3,3)--(11.7,3);
\draw[red,ultra thick] (10.3,4)--(11.7,4);
\draw[red,ultra thick] (10.3,5)--(11.7,5);
\draw[red,ultra thick] (10.3,6)--(11.7,6);

\draw[dashed,ultra thick] (10.3,6)--(11.7,4);
\draw[dashed,ultra thick] (10.3,5)--(11.7,3);
\draw[dashed,ultra thick] (10.3,4)--(11.7,2);
\draw[dashed,ultra thick] (10.3,3)--(11.7,1);
\draw[dashed,ultra thick] (10.3,2)--(11.7,0);

\node[] at (11,-1) {$G_2(n,\alpha)\cup F(n,\alpha)$};

\end{tikzpicture}
\caption{Illustration of some aspects of the construction when $n_a>n_b$, $b=a-1$, $n$ even.}
\label{fig4_3}
\end{figure}

\subsubsection{When $n_a\geq n_b+2$ and $b=a+1$} \label{subsec-construction-d<=n/s-n-even-4}

\begin{theorem} \label{thm-alpha-not-divide-n/2-n_a>=n_b+2-b=a+1}
Fix $n, d$ and $\alpha$ with $n \geq 2$ and even, $1 \leq \alpha \leq n/2$, $n/2$ not divisible by $\alpha$, $n_a\geq n_b+2$, $b=a+1$, and $a+b-1 \leq d \leq n/2+a-1$. There is an $(n,d,\alpha)$-graph $G(n,d,\alpha)$ that has $Z(n/2,\alpha)$ as an induced subgraph.  
\end{theorem}

\begin{proof}
As usual, we start with $G_1(n,\alpha)$. In this case, we adjust $G_1(n,\alpha)$ to $G_2(n,\alpha)$ by adding a perfect matching between $X_i$ and $Y_{i-1}$ for each $i=n_b+2,\ldots, \alpha-n_b$, and between $X_{n_b+1}$ and $Y_{\alpha-n_b}$. Note that $G_2(n,\alpha)$ is $(a+b-1)$-regular.

For this case an appropriate choice for $F(n,\alpha)$ is the graph on $X\cup Y$ that consists of a perfect matching between $X_i$ and $Y_{i-n_a}$ for $i=n_a+1,\ldots,\alpha$. As before, let $G_3(n,\alpha)$ be the graph on $X\cup Y$ whose edges are all the pairs $\{x,y\}$ with $x\in X$ and $y \in Y$ such that $\{x,y\}$ is not an edge of $G_2(n,\alpha)$ or $F(n,\alpha)$. Observe that $G_3(n,\alpha)$ is an $((n/2)-b)$-regular bipartite graph, so it follows that for every $d'$ satisfying $0\leq d'\leq n/2-b$, $G_3(n,\alpha)$ has a $d'$-regular subgraph. Select one with $d'=d-(a+b-1)$, and call it $G_4(n,d,\alpha)$. Then $G(n,d,\alpha)=G_2(n,\alpha)\cup G_4(n,d,\alpha)$ is an $(n,d,\alpha)$-graph that has $Z(n/2,\alpha)$ as an induced subgraph.
\end{proof}

Figure \ref{fig4_4} shows some of the intermediate steps of this construction. This specific figure represents the situation when $n_a=5$ and $n_b=2$.

\begin{figure}[ht!] 
\centering
\begin{tikzpicture}
\node[] at (0,0) {$K_b$};
\node[] at (0,1) {$K_b$};
\node[] at (0,2) {$K_a$};
\node[] at (0,3) {$K_a$};
\node[] at (0,4) {$K_a$};
\node[] at (0,5) {$K_a$};
\node[] at (0,6) {$K_a$};

\node[] at (2,0) {$K_a$};
\node[] at (2,1) {$K_a$};
\node[] at (2,2) {$K_a$};
\node[] at (2,3) {$K_a$};
\node[] at (2,4) {$K_a$};
\node[] at (2,5) {$K_b$};
\node[] at (2,6) {$K_b$};

\draw (0.3,4)--(1.7,2);
\draw (0.3,3)--(1.7,4);
\draw (0.3,2)--(1.7,3);

\draw[red,ultra thick] (0.3,0)--(1.7,0);
\draw[red,ultra thick] (0.3,1)--(1.7,1);
\draw[red,ultra thick] (0.3,2)--(1.7,2);
\draw[red,ultra thick] (0.3,3)--(1.7,3);
\draw[red,ultra thick] (0.3,4)--(1.7,4);
\draw[red,ultra thick] (0.3,5)--(1.7,5);
\draw[red,ultra thick] (0.3,6)--(1.7,6);

\node[] at (1,-1) {$G_2(n,\alpha)$};

\node[] at (10,0) {$K_b$};
\node[] at (10,1) {$K_b$};
\node[] at (10,2) {$K_a$};
\node[] at (10,3) {$K_a$};
\node[] at (10,4) {$K_a$};
\node[] at (10,5) {$K_a$};
\node[] at (10,6) {$K_a$};

\node[] at (12,0) {$K_a$};
\node[] at (12,1) {$K_a$};
\node[] at (12,2) {$K_a$};
\node[] at (12,3) {$K_a$};
\node[] at (12,4) {$K_a$};
\node[] at (12,5) {$K_b$};
\node[] at (12,6) {$K_b$};

\draw (10.3,4)--(11.7,2);
\draw (10.3,3)--(11.7,4);
\draw (10.3,2)--(11.7,3);
\draw[dashed,ultra thick] (10.3,1)--(11.7,6);
\draw[dashed,ultra thick] (10.3,0)--(11.7,5);

\draw[red,ultra thick] (10.3,0)--(11.7,0);
\draw[red,ultra thick] (10.3,1)--(11.7,1);
\draw[red,ultra thick] (10.3,2)--(11.7,2);
\draw[red,ultra thick] (10.3,3)--(11.7,3);
\draw[red,ultra thick] (10.3,4)--(11.7,4);
\draw[red,ultra thick] (10.3,5)--(11.7,5);
\draw[red,ultra thick] (10.3,6)--(11.7,6);

\node[] at (11,-1) {$G_2(n,\alpha)\cup F(n,\alpha)$};

\end{tikzpicture}
\caption{Illustration of some aspects of the construction when $b=a+1,n_a\geq n_b+2$, $n$ even.}
\label{fig4_4}
\end{figure}
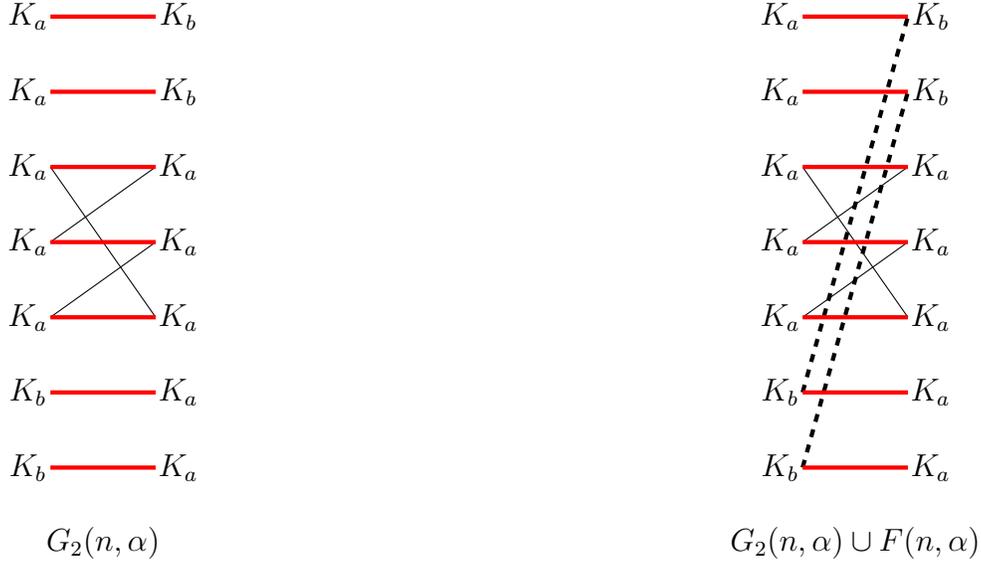

\subsubsection{When $n_a=n_b+1$, $b=a+1$} \label{subsec-construction-d<=n/s-n-even-5}
  
\begin{theorem} \label{thm-alpha-not-divide-n/2-n_a=n_b+1-b=a+1}
Fix $n, d$ and $\alpha$ with $n \geq 2$ and even, $1 \leq \alpha \leq n/2$, $n/2$ not divisible by $\alpha$, $n_a = n_b+1$, $b=a+1$, and $a+b \leq d \leq n/2+b-2$. There is an $(n,d,\alpha)$-graph $G(n,d,\alpha)$ that has $Z(n/2,\alpha)$ as an induced subgraph.  
\end{theorem}

\begin{proof}
As in all the previous cases we start with $G_1(n,\alpha)$ as our base graph. Each vertex in $G_1(n,\alpha)$ has degree $a+b-1$, except for the vertices in $X_{n_a}$ and $Y_{n_a}$, which have degree $2a-1$, precisely one less. Unfortunately, we cannot simply add a perfect matching between $X_{n_a}$ and $Y_{n_a}$ to increase the degree of those vertices by 1 since every edge between them has already been used in the construction of $G_1(n,\alpha)$. Because of this, to make our modified $G_2(n,\alpha)$ graph regular we are forced to increase the degree of the vertices in $X_{n_a}$ and $Y_{n_a}$ by 2 and increase the degree of all other vertices in $G_1(n,\alpha)$ by 1. We detail the necessary additional edges below.

First, we put a perfect matching between $X_i$ and $Y_{\alpha+1-i}$ for $i=1,\ldots, n_a-2$, and between $X_{j+n_a+1}$ and $Y_j$ for $j=1,\ldots,n_a-2$. Additionally, we add a perfect matching between $X_{n_a-1}$ and $Y_{n_a}$ and between $X_{n_a}$ and $Y_{n_a+1}$. 

To make the remaining edges easier to describe we split the vertex set $X_{n_a+1}$ into two subsets: one consisting of a single vertex (call it $X_{n_a+1,1}$), and the other consisting of the remaining $a$ vertices (call it $X_{n_a+1,a}$). Similarly, we split $Y_{n_a-1}$ into two subsets $Y_{n_a+1,1}$ (of size one) and $Y_{n_a+1,a}$ (of size $a$). Now we add a perfect matching between $X_{n_a}$ and $Y_{n_a+1,a}$, and a perfect matching between $X_{n_a+1,a}$ and $Y_{n_a}$. Finally, we add the edge between $X_{n_a+1,1}$ and $Y_{n_a+1,1}$. The addition of all these edges produces a graph $G_2(n,\alpha)$ that is $(a+b)$-regular.

Now define $F(n,\alpha)$ to be the graph on $X\cup Y$ which consists of a perfect matching between $X_i$ and $Y_{\alpha+1-i}$ for $i=n_a+1,\ldots,\alpha$. 

As usual, now let $G_3(n,\alpha)$ be the graph on $X\cup Y$ whose edges are all the pairs $\{x,y\}$ with $x\in X$ and $y \in Y$ such that $\{x,y\}$ is not an edge of $G_2(n,\alpha)$ or $F(n,\alpha)$. Observe that $G_3(n,\alpha)$ is a $((n/2)-a-2)$-regular bipartite graph, so it follows that for every $d'$ satisfying $0\leq d'\leq (n/2)-a-2$, $G_3(n,\alpha)$ has a $d'$-regular subgraph. Select one with $d'=d-(a+b)$ and call it $G_4(n,d,\alpha)$. Then $G(n,d,\alpha)=G_2(n,\alpha)\cup G_4(n,d,\alpha)$ is an $(n,d,\alpha)$-graph that has $Z(n/2,\alpha)$ as an induced subgraph.
\end{proof}

Figure \ref{fig4_5} shows some of the intermediate steps of this construction. This specific figure represents the situation when $n_a=3$ and $n_b=2$.

\begin{figure}[ht!] 
\centering
\begin{tikzpicture}
\node[] at (0,0) {$K_b$};
\node[blue] at (0,1) {$K_a$};
\node[blue] at (0,2) {$K_1$};
\node[] at (0,3) {$K_a$};
\node[] at (0,4.5) {$K_a$};
\node[] at (0,6) {$K_a$};

\node[] at (2,0) {$K_a$};
\node[] at (2,1.5) {$K_a$};
\node[] at (2,3) {$K_a$};
\node[blue] at (2,4) {$K_1$};
\node[blue] at (2,5) {$K_a$};
\node[] at (2,6) {$K_b$};

\draw (0.3,0)--(1.7,6);
\draw (0.3,1)--(1.7,3);
\draw (0.3,2)--(1.7,4);
\draw (0.3,3)--(1.7,1.5);
\draw (0.3,3)--(1.7,5);
\draw (0.3,4.5)--(1.7,3);
\draw (0.3,6)--(1.7,0);

\draw[red,ultra thick] (0.3,0)--(1.7,0);
\draw[red,ultra thick] (0.3,1)--(1.7,1.5);
\draw[red,ultra thick] (0.3,2)--(1.7,1.5);
\draw[red,ultra thick] (0.3,3)--(1.7,3);
\draw[red,ultra thick] (0.3,4.5)--(1.7,4);
\draw[red,ultra thick] (0.3,4.5)--(1.7,5);
\draw[red,ultra thick] (0.3,6)--(1.7,6);
\draw[red,ultra thick] (0,1.3)--(0,1.7);
\draw[red,ultra thick] (2,4.3)--(2,4.7);

\node[] at (1,-1) {$G_2(n,\alpha)$};

\node[] at (10,0) {$K_b$};
\node[blue] at (10,1) {$K_a$};
\node[blue] at (10,2) {$K_1$};
\node[] at (10,3) {$K_a$};
\node[] at (10,4.5) {$K_a$};
\node[] at (10,6) {$K_a$};

\node[] at (12,0) {$K_a$};
\node[] at (12,1.5) {$K_a$};
\node[] at (12,3) {$K_a$};
\node[blue] at (12,4) {$K_1$};
\node[blue] at (12,5) {$K_a$};
\node[] at (12,6) {$K_b$};

\draw (10.3,0)--(11.7,6);
\draw (10.3,1)--(11.7,3);
\draw (10.3,2)--(11.7,4);
\draw (10.3,3)--(11.7,1.5);
\draw (10.3,3)--(11.7,5);
\draw (10.3,4.5)--(11.7,3);
\draw (10.3,6)--(11.7,0);
\draw[dashed,ultra thick] (10.3,0.2)--(11.7,6.2);
\draw[dashed,ultra thick] (10.3,1.5)--(11.7,4.5);
\draw[red,ultra thick] (10.3,0)--(11.7,0);
\draw[red,ultra thick] (10.3,1)--(11.7,1.5);
\draw[red,ultra thick] (10.3,2)--(11.7,1.5);
\draw[red,ultra thick] (10.3,3)--(11.7,3);
\draw[red,ultra thick] (10.3,4.5)--(11.7,4);
\draw[red,ultra thick] (10.3,4.5)--(11.7,5);
\draw[red,ultra thick] (10.3,6)--(11.7,6);

\node[] at (11,-1) {$G_2(n,\alpha)\cup F_2(n,\alpha)$};

\end{tikzpicture}
\caption{Illustration of some aspects of the construction when $b=a+1,n_a=n_b+1$, $n$ even.}
\label{fig4_5}
\end{figure}
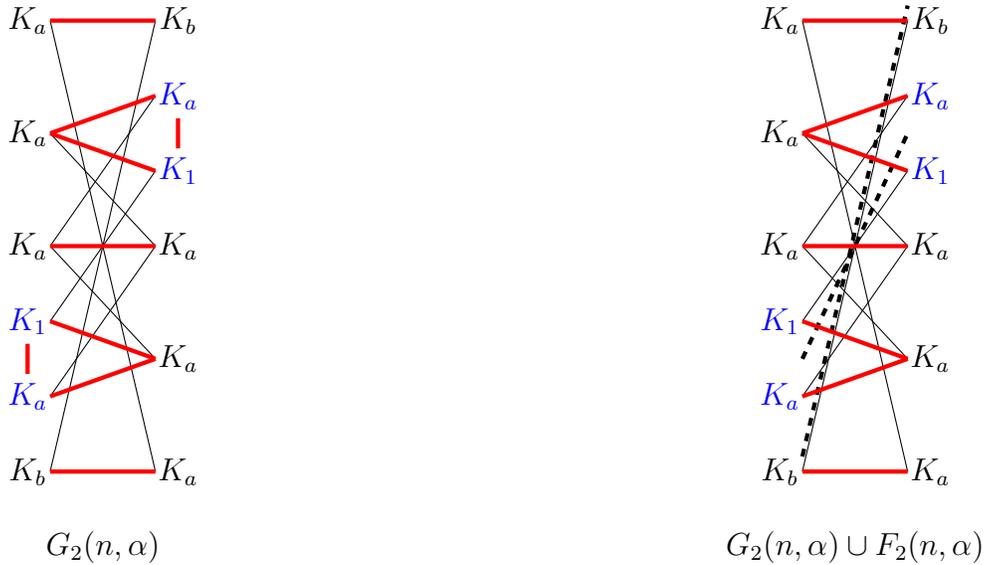

\subsubsection{Putting it all together (even $n$)}

Combining all of the previous cases together (specifically, Theorems \ref{thm-alpha-divides-n/2}, \ref{thm-alpha-not-divide-n/2-n_a=n_b}, \ref{thm-alpha-not-divide-n/2-n_a>n_b-b=a-1}, \ref{thm-alpha-not-divide-n/2-n_a>=n_b+2-b=a+1} and \ref{thm-alpha-not-divide-n/2-n_a=n_b+1-b=a+1}), we have the following theorem. Note that we could formulate a slightly more refined theorem (for example replacing $\lf n/(2\alpha) \rf +\lc n/(2\alpha) \rc$ below with $\lf n/(2\alpha) \rf +\lc n/(2\alpha) \rc -1$ for certain choices of $(n,d,\alpha)$), but what appears below is more than enough for our purposes. 

\begin{theorem} \label{even-n-lb}
For $n, d$ and $\alpha$ with $n \geq 2$ and even, $1 \leq \alpha \leq n/2$, and 
$$
\lf n/(2\alpha) \rf +\lc n/(2\alpha) \rc \leq d \leq n/2+\lf n/(2\alpha) \rf -1,
$$ 
there is an $(n,d,\alpha)$-graph $G(n,d,\alpha)$ that has $Z(n/2,\alpha)$ as an induced subgraph.  
\end{theorem}

\subsection{The construction for $n$ odd} \label{subsec-constructions-d<=n/2-n-odd}

We will not need to consider $n$ odd in order to derive {\bf LB1}, but for completeness we present a family of constructions for this case. Note that an $(n,d)$-graph with $n$ odd must have $d$ even. With this in mind, we prove the following theorem. 
\begin{theorem} \label{odd-n-lb}
For $n, d$ and $\alpha$ with $n \geq 2$ and odd, $d\geq 2$ even, $1 \leq \alpha \leq (n-1)/2$, and 
$$
2\left(\lf \frac{n-1}{2\alpha} \rf +\lc \frac{n-1}{2\alpha} \rc\right) \leq d \leq \frac{n-1}{2}+\lf \frac{n-1}{2\alpha} \rf -1,
$$
there is an $(n,d,\alpha)$-graph $G(n,d,\alpha)$ that has $Z(\lf n/2 \rf,\alpha)$ as an induced subgraph.  
\end{theorem}

The basic idea for the proof is that for a given $(n,d,\alpha)$ we start with the corresponding construction in Section \ref{subsec-constructions-d<=n/2-n-even} for the triple $(n-1,d,\alpha)$. The last step in that construction was adding a $d'$-regular bipartite graph with partite classes being the vertex sets of two copies of $Z(\lf n/2 \rf,\alpha)$ that the construction began with. We remove $d/2$ of the these edges, chosen so that there is a pair $(X_i,Y_i)$ (a pair of corresponding cliques from the two copies of $Z(\lf n/2 \rf,\alpha)$ that the construction began with) such that a perfect matching between $X_i$ and $Y_i$ is included among the removed edges. We then introduce an $n$th vertex, and join it to the end-vertices of these $d/2$ removed edges, to make the entire graph $d$-regular. The resulting graph is $d$-regular, has $Z(\lf n/2 \rf, \alpha)$ as an induced subgraph, and still has largest independent size $\alpha$. 

The lower bound on $d$ in Theorem \ref{odd-n-lb} is chosen to achieve three things: 
\begin{itemize}
\item to ensure that the appropriate construction from Section \ref{subsec-constructions-d<=n/2-n-even} is valid with parameters $(n-1,d,\alpha)$;
\item to ensure that in the last step of that construction we have $d' \geq 1$ (so we actually have $d$ edges to remove); and
\item to ensure that $X_i$ and $Y_i$ are small enough that we can remove a perfect matching between $X_i$ and $Y_i$ when we remove $d$ edges. 
\end{itemize}
The factor of 2 in the lower bound on $d$ in Theorem \ref{odd-n-lb}, that is absent from the lower bound on $d$ in Theorem \ref{even-n-lb}, is needed to ensure that the last of these conditions hold.

\begin{proof} (Proof of Theorem \ref{odd-n-lb}.)
We start with an $(n-1,d,\alpha)$-graph $G(n-1,d,\alpha)$ on a vertex set $X \cup Y$ obtained via the constructions in Section \ref{subsec-constructions-d<=n/2-n-even}, with corresponding $G_2(n-1,\alpha)$ and $G_4(n-1,d,\alpha)$. We have $d'\geq 1$ since 
$$
d\geq 2\left(\lf \frac{n}{2\alpha}\rf +\lc \frac{n}{2\alpha} \rc\right) \geq \lf \frac{n-1}{2\alpha} \rf +\lc \frac{n-1}{2\alpha} \rc+1,
$$ 
so, since also $d < n$, we can select a subset $E$ of $E(G_4(n-1,d,\alpha))$ containing $d/2$ edges. 

Moreover, we can choose the edges in $E$ so that each vertex of $X_1$ and each vertex of $Y_1$ is an endvertex of an edge in $E$, since $d/2\geq \lf (n-1)/2\alpha \rf +\lc (n-1)/2\alpha \rc$. This precision is necessary so that the independence number of the graph we are constructing remains equal to $\alpha$.

Next, we add an $n$th vertex, join it to every vertex in $E$, and delete every edge in $E$. Observe that the resulting graph, which we call $G(n,d,\alpha)$, is $d$-regular, and has independence number $\alpha$. To see this latter fact, note that before adding the $n$th vertex, the construction included as a subgraph a disjoint union of $\alpha$ cliques (specifically, on vertex sets $X_1\cup Y_1, \ldots, X_{\alpha}\cup Y_{\alpha}$), and that the $n$th vertex is adjacent to everything in $X_1 \cup Y_1$; so $G(n,d,\alpha)$ has as a subgraph a disjoint union of $\alpha$ cliques. Further, $G(n,d,\alpha)$ evidently has $Z(\lf n/2 \rf,\alpha)$ as an induced subgraph.
\end{proof}

Figure \ref{4_odd} shows the steps of this construction for the specific case when $n=13$, $\alpha=6$, and $d=3$. Here the edges of $G_4(n-1,\alpha)$ are colored blue, the edges of $E$ are orange, and the new edges to the $n$th vertex are green. The steps are similar for other values of $n$, $d$, and $\alpha$.

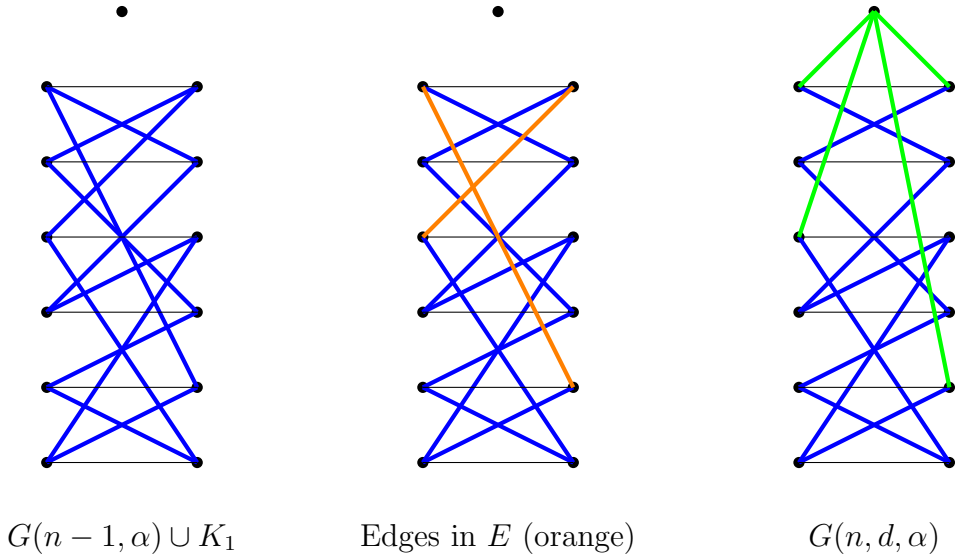
\begin{figure}[ht!] 
\centering
\begin{tikzpicture}

\node[] at (0,1) [circle,fill,inner sep=1.5pt]{};
\node[] at (0,2) [circle,fill,inner sep=1.5pt]{};
\node[] at (0,3) [circle,fill,inner sep=1.5pt]{};
\node[] at (0,4) [circle,fill,inner sep=1.5pt]{};
\node[] at (0,5) [circle,fill,inner sep=1.5pt]{};
\node[] at (0,6) [circle,fill,inner sep=1.5pt]{};

\node[] at (2,1) [circle,fill,inner sep=1.5pt]{};
\node[] at (2,2) [circle,fill,inner sep=1.5pt]{};
\node[] at (2,3) [circle,fill,inner sep=1.5pt]{};
\node[] at (2,4) [circle,fill,inner sep=1.5pt]{};
\node[] at (2,5) [circle,fill,inner sep=1.5pt]{};
\node[] at (2,6) [circle,fill,inner sep=1.5pt]{};

\node[] at (1,7) [circle,fill,inner sep=1.5pt]{};

\draw[blue,ultra thick] (0,1)--(2,4);
\draw[blue,ultra thick] (0,2)--(2,1);
\draw[blue,ultra thick] (0,3)--(2,5);
\draw[blue,ultra thick] (0,4)--(2,6);
\draw[blue,ultra thick] (0,5)--(2,3);
\draw[blue,ultra thick] (0,6)--(2,2);
\draw[blue,ultra thick] (0,1)--(2,2);
\draw[blue,ultra thick] (0,2)--(2,3);
\draw[blue,ultra thick] (0,3)--(2,4);
\draw[blue,ultra thick] (0,4)--(2,1);
\draw[blue,ultra thick] (0,5)--(2,6);
\draw[blue,ultra thick] (0,6)--(2,5);

\draw (0,1)--(2,1);
\draw (0,2)--(2,2);
\draw (0,3)--(2,3);
\draw (0,4)--(2,4);
\draw (0,5)--(2,5);
\draw (0,6)--(2,6);

\node[] at (1,0) {$G(n-1,\alpha)\cup K_1$};

\node[] at (5,1) [circle,fill,inner sep=1.5pt]{};
\node[] at (5,2) [circle,fill,inner sep=1.5pt]{};
\node[] at (5,3) [circle,fill,inner sep=1.5pt]{};
\node[] at (5,4) [circle,fill,inner sep=1.5pt]{};
\node[] at (5,5) [circle,fill,inner sep=1.5pt]{};
\node[] at (5,6) [circle,fill,inner sep=1.5pt]{};

\node[] at (7,1) [circle,fill,inner sep=1.5pt]{};
\node[] at (7,2) [circle,fill,inner sep=1.5pt]{};
\node[] at (7,3) [circle,fill,inner sep=1.5pt]{};
\node[] at (7,4) [circle,fill,inner sep=1.5pt]{};
\node[] at (7,5) [circle,fill,inner sep=1.5pt]{};
\node[] at (7,6) [circle,fill,inner sep=1.5pt]{};

\node[] at (6,7) [circle,fill,inner sep=1.5pt]{};

\draw[blue,ultra thick] (5,1)--(7,4);
\draw[blue,ultra thick] (5,2)--(7,1);
\draw[blue,ultra thick] (5,3)--(7,5);
\draw[blue,ultra thick] (5,5)--(7,3);
\draw[blue,ultra thick] (5,1)--(7,2);
\draw[blue,ultra thick] (5,2)--(7,3);
\draw[blue,ultra thick] (5,3)--(7,4);
\draw[blue,ultra thick] (5,4)--(7,1);
\draw[blue,ultra thick] (5,5)--(7,6);
\draw[blue,ultra thick] (5,6)--(7,5);

\draw (5,1)--(7,1);
\draw (5,2)--(7,2);
\draw (5,3)--(7,3);
\draw (5,4)--(7,4);
\draw (5,5)--(7,5);
\draw (5,6)--(7,6);

\draw[orange,ultra thick] (5,4)--(7,6);
\draw[orange,ultra thick] (5,6)--(7,2);

\node[] at (6,0) {Edges in $E$ (orange)};

\node[] at (10,1) [circle,fill,inner sep=1.5pt]{};
\node[] at (10,2) [circle,fill,inner sep=1.5pt]{};
\node[] at (10,3) [circle,fill,inner sep=1.5pt]{};
\node[] at (10,4) [circle,fill,inner sep=1.5pt]{};
\node[] at (10,5) [circle,fill,inner sep=1.5pt]{};
\node[] at (10,6) [circle,fill,inner sep=1.5pt]{};

\node[] at (12,1) [circle,fill,inner sep=1.5pt]{};
\node[] at (12,2) [circle,fill,inner sep=1.5pt]{};
\node[] at (12,3) [circle,fill,inner sep=1.5pt]{};
\node[] at (12,4) [circle,fill,inner sep=1.5pt]{};
\node[] at (12,5) [circle,fill,inner sep=1.5pt]{};
\node[] at (12,6) [circle,fill,inner sep=1.5pt]{};

\node[] at (11,7) [circle,fill,inner sep=1.5pt]{};

\draw[blue,ultra thick] (10,1)--(12,4);
\draw[blue,ultra thick] (10,2)--(12,1);
\draw[blue,ultra thick] (10,3)--(12,5);
\draw[blue,ultra thick] (10,5)--(12,3);
\draw[blue,ultra thick] (10,1)--(12,2);
\draw[blue,ultra thick] (10,2)--(12,3);
\draw[blue,ultra thick] (10,3)--(12,4);
\draw[blue,ultra thick] (10,4)--(12,1);
\draw[blue,ultra thick] (10,5)--(12,6);
\draw[blue,ultra thick] (10,6)--(12,5);

\draw (10,1)--(12,1);
\draw (10,2)--(12,2);
\draw (10,3)--(12,3);
\draw (10,4)--(12,4);
\draw (10,5)--(12,5);
\draw (10,6)--(12,6);

\draw[green,ultra thick] (10,4)--(11,7);
\draw[green,ultra thick] (10,6)--(11,7);
\draw[green,ultra thick] (12,2)--(11,7);
\draw[green,ultra thick] (12,6)--(11,7);

\node[] at (11,0) {$G(n,d,\alpha)$};

\end{tikzpicture}
\caption{Illustration of some aspects of the construction when $n$ is odd.}
\label{4_odd}
\end{figure}

\subsection{Proof of {\bf LB1}} \label{subsec-proof-of-LB1}

We are now in a position to very quickly establish {\bf LB1}. Let $N$ be the set of multiples of $4$, and let $d_n=n/4$. Let $\alpha_n=c_{\rm ind}n + o(n)$ (the specific choice here is quite arbitrary). By removing a finite initial segment from $N$ if necessary, we can assume that 
$$
d_n \geq \lf \frac{n}{2\alpha_n} \rf + \lc \frac{n}{2\alpha_n} \rc.
$$
This is because $\lf n/(2\alpha_n) \rf + \lc n/(2\alpha_n) \rc = O(1)$ while $d_n = \omega(1)$.

By Theorem \ref{even-n-lb} there exists, for each $n \in N$, an $(n,d_n,\alpha_n)$-graph $G_n$ that has $Z(n/2, \alpha_n)$ as an induced subgraph, and that therefore satisfies $$
i(G_n) \geq i(Z(n/2, \alpha))=k(c_{\rm ind})^{n+o(n)},
$$ 
the last equality coming from Lemma \ref{lem-z-asy}.

\section{Constructing regular graphs with Zykov graphs as subgraphs --- the case $d \geq n/2$} \label{sec-constructions-d>=n/2}

In this section we establish \textbf{LB3}, the range $d \geq n/2$ (and $\alpha \leq n-d$). As with Section \ref{sec-constructions-d<=n/2}, we would ideally like to establish that for all possible triples $(n,d,\alpha)$ (triples for which an $(n,d,\alpha)$-graph actually exists) there is one with $Z(n-d,\alpha)$ as an induced subgraph. We don't quite achieve this goal, but we cover enough ground to deduce \textbf{LB3}. 

We have two constructions, covering different ranges of $d$. In Section \ref{subsection-d>2n/3-construct} we consider $2n/3 \lessapprox d < n$ and in Section \ref{d=2n/3case} we describe a construction valid for $n/2 \lessapprox d \lessapprox n$. This second construction is enough to establish {\bf LB3} for all claimed values of $c_{\rm ind}$ and $c_{\rm deg}$. We include the earlier construction both because it closes a small gap near $d=n$, and because it is quite simple.     

Both constructions will start with two disjoint copies of $Z(n-d,\alpha)$, similar to the base graph $G_1(n,\alpha)$ defined in Section \ref{subsec-constructions-d<=n/2-n-even}, together with a set of $2d-n$ vertices. We pair up the cliques in the two copies of $Z(n-d,\alpha)$, except this time instead of placing complete bipartite graphs between each pair of corresponding cliques in the two copies of $Z(n-d,\alpha)$, we put a perfect matching instead. Between all other pairs of cliques in the two copies of $Z(n-d,\alpha)$ we place complete bipartite graphs. As long as $\alpha > 1$ (which, as observed earlier, we can assume) this still ensures that the largest independent set (thus far) is of size $\alpha$. Inside one of the two copies of $Z(n-d,\alpha)$ we commit to adding no more edges, thus ensuring that $Z(n-d,\alpha)$ appears as an induced subgraph. Finally we add edges inside the other copy of $Z(n-d,\alpha)$, inside the remaining set of $2d-n$ vertices, and between these latter two sets of vertices to make the graph regular, all the while making sure that we do not create an independent set of size larger than $\alpha$.  

\subsection{The base graph $G_1(n,d,\alpha)$} \label{subsec-base-graph-2}

We now establish some notation, and formalise our construction of the base graph $G_1(n,d,\alpha)$. Let $n$ and $d$ satisfy $d \geq n/2$. Then $Z(n-d,\alpha)$ is the disjoint union of $n_a$ cliques of order $a$ and $n_b$ cliques of order $b$ with $|a-b|=1$, unless $\alpha$ divides $n-d$, in which case $a=b$ and it will be convenient to take $n_a=\alpha$ and $n_b=0$. Otherwise, we assume (without loss of generality) that $b > a$. We construct a graph $G_1(n,d,\alpha)$ on vertex set $W\cup X\cup Y$ where $W$, $X$, and $Y$ are disjoint sets, with $W$ of size $2d-n$ and each of $X$ and $Y$ of size $n-d$, as follows:

\begin{itemize}
\item The vertex set $X$ induces a copy of $Z(n-d,\alpha)$ with cliques $X_1,\ldots,X_{n_a}$ of order $a$ and 
$X_{n_a+1},\ldots,X_\alpha$ of order $b$. 
\item The vertex set $Y$ induces a copy of $Z(n-d,\alpha)$ with cliques $Y_1,\ldots,Y_{n_a}$ of order $a$ and $Y_{n_a+1},\ldots,Y_\alpha$ of order $b$. 
\item For each $i=1,\ldots,\alpha$ there is a perfect matching between $X_i$ and $Y_i$. 
\item For each $i\neq j$ there is a complete bipartite graph between $X_i$ and $Y_j$. 
\end{itemize}
Note that the subgraph of $G_1(n,d,\alpha)$ induced by $X \cup Y$ is an $(n-d)$-regular graph with independence number $\alpha$. To see that the independence number is $\alpha$, note that there is an independent set of size $\alpha$ in each of the copies of $Z(n-d,\alpha)$, and there can be no larger independent set contained completely within $X$ or completely within $Y$. Suppose that an independent set includes vertices from both $X$ and $Y$. If $x \in X_i$ is in the independent set, then the vertices from $Y$ must come from $Y_i$, and there can be at most one such, say $y$. It is an easy check that $\{x,y\}$ cannot be extended as an independent set inside $X \cup Y$. Finally, recall from the introduction that we are assuming $\alpha \geq 2$.  

\subsection{Preliminary regular graph constructions}

To aid in the two constructions of $G(n,d,\alpha)$, in addition to $G_1(n,d,\alpha)$ we will make use of two closely related families of regular graphs, defined below. In each case the graphs are described on vertex set $\{v_0, \ldots, v_{m-1}\}$ (of size $m$) and are $r$-regular. 

\begin{itemize}
\item $R(m,r)$: 
\begin{itemize}
\item If $r$ is even, then $v_iv_j$ is an edge if and only if $i-j \in [1,r/2] \cup [-r/2,-1] \pmod{m}$, and
\item if $r$ is odd (and so $m$ is even), then $v_iv_j$ is an edge if and only if either $i-j \in [1,(r-1)/2] \cup [-(r-1)/2,-1] \pmod{m}$ or $i-j = m/2 \pmod{m}$.
\end{itemize}
In other words, with the vertices arranged cyclically each vertex is adjacent to an equal number of vertices immediately to its right and immediately to its left, and also (if necessary) to a single vertex immediately opposite it.   
\item $R(m,r,g)$:
\begin{itemize}
\item
If $r$ is even, then $v_iv_j$ is an edge if and only if $i-j \in [g+1,g+r/2] \cup [-g-r/2,-g-1] \pmod{m}$, and
\item if $r$ is odd (and so $m$ is even), then $v_iv_j$ is an edge if and only if either $i-j \in [g+1,g+(r-1)/2] \cup [-g-(r-1)/2,-g-1] \pmod{m}$ or $i-j = m/2 \pmod{m}$.
\end{itemize}
In other words, with the vertices arranged cyclically each vertex is adjacent to an equal number of vertices immediately to its right and immediately to its left, not including the $g$ vertices nearest it on either side, and also (if necessary) to a single vertex immediately opposite it. Note that a necessary condition here for this graph to actually be $r$-regular is that $g+\lf r/2 \rf \leq m/2 -1$. 
\end{itemize}

Note that $R(m,r,0)$ is simply $R(m,r)$.
Examples of these two constructions for small $m,r$ and $g$ are shown in Figure \ref{R_constructions}.
\begin{figure}[ht!] 
\centering
\begin{tikzpicture}
\node (a) at (0,0)
         {
            \begin{tikzpicture}
               \coordinate (a) at ({2*cos(0)}, {2*sin(0)});
                \coordinate (b) at ({2*cos(30)}, {2*sin(30)});
                \coordinate (c) at ({2*cos(60)}, {2*sin(60)});
                \coordinate (d) at ({2*cos(90)}, {2*sin(90)});
                \coordinate (e) at ({2*cos(120)}, {2*sin(120)});
                \coordinate (f) at ({2*cos(150)}, {2*sin(150)});
                \coordinate (g) at ({2*cos(180)}, {2*sin(180)});
                \coordinate (h) at ({2*cos(210)}, {2*sin(210)});
                \coordinate (i) at ({2*cos(240)}, {2*sin(240)});
                \coordinate (j) at ({2*cos(270)}, {2*sin(270)});
                \coordinate (k) at ({2*cos(300)}, {2*sin(300)});
                \coordinate (l) at ({2*cos(330)}, {2*sin(330)});

                \draw[fill=black] (a) circle (3pt);
                \draw[fill=black] (b) circle (3pt);
                \draw[fill=black] (c) circle (3pt);
                \draw[fill=black] (d) circle (3pt);
                \draw[fill=black] (e) circle (3pt);
                \draw[fill=black] (f) circle (3pt);
                \draw[fill=black] (g) circle (3pt);
                \draw[fill=black] (h) circle (3pt);
                \draw[fill=black] (i) circle (3pt);
                \draw[fill=black] (j) circle (3pt);
                \draw[fill=black] (k) circle (3pt);
                \draw[fill=black] (l) circle (3pt);

                \draw[black,ultra thick] (a)--(b);
                \draw[black,ultra thick] (a)--(c);
                \draw[black,ultra thick] (a)--(l);
                \draw[black,ultra thick] (a)--(k);
                \draw[black,ultra thick] (a)--(g);
                
                \node[] at (0,-3) {Some edges of $R(12,5)$};
                
               \node[] at ({2.5*cos(0)}, {2.5*sin(0)}) {$v_0$};;
                \node[] at ({2.5*cos(30)}, {2.5*sin(30)}) {$v_1$};;
                \node[] at ({2.5*cos(60)}, {2.5*sin(60)}) {$v_2$};;
                \node[] at ({2.5*cos(90)}, {2.5*sin(90)}) {$v_3$};;
                \node[] at ({2.5*cos(120)}, {2.5*sin(120)}) {$v_4$};;
               \node[] at ({2.5*cos(150)}, {2.5*sin(150)}) {$v_5$};;
                \node[] at ({2.5*cos(180)}, {2.5*sin(180)}) {$v_6$};;
               \node[] at ({2.5*cos(210)}, {2.5*sin(210)}) {$v_7$};;
               \node[] at ({2.5*cos(240)}, {2.5*sin(240)}) {$v_8$};;
                \node[] at ({2.5*cos(270)}, {2.5*sin(270)}) {$v_9$};;
               \node[] at ({2.5*cos(300)}, {2.5*sin(300)}) {$v_{10}$};;
                \node[] at ({2.5*cos(330)}, {2.5*sin(330)}) {$v_{11}$};;
                            \end{tikzpicture}
         };
\node (c) at (7,0)
         {
            \begin{tikzpicture}
               \coordinate (a) at ({2*cos(0)}, {2*sin(0)});
                \coordinate (b) at ({2*cos(30)}, {2*sin(30)});
                \coordinate (c) at ({2*cos(60)}, {2*sin(60)});
                \coordinate (d) at ({2*cos(90)}, {2*sin(90)});
                \coordinate (e) at ({2*cos(120)}, {2*sin(120)});
                \coordinate (f) at ({2*cos(150)}, {2*sin(150)});
                \coordinate (g) at ({2*cos(180)}, {2*sin(180)});
                \coordinate (h) at ({2*cos(210)}, {2*sin(210)});
                \coordinate (i) at ({2*cos(240)}, {2*sin(240)});
                \coordinate (j) at ({2*cos(270)}, {2*sin(270)});
                \coordinate (k) at ({2*cos(300)}, {2*sin(300)});
                \coordinate (l) at ({2*cos(330)}, {2*sin(330)});

                \draw[fill=black] (a) circle (3pt);
                \draw[fill=black] (b) circle (3pt);
                \draw[fill=black] (c) circle (3pt);
                \draw[fill=black] (d) circle (3pt);
                \draw[fill=black] (e) circle (3pt);
                \draw[fill=black] (f) circle (3pt);
                \draw[fill=black] (g) circle (3pt);
                \draw[fill=black] (h) circle (3pt);
                \draw[fill=black] (i) circle (3pt);
                \draw[fill=black] (j) circle (3pt);
                \draw[fill=black] (k) circle (3pt);
                \draw[fill=black] (l) circle (3pt);

                \draw[black,ultra thick] (a)--(d);
                \draw[black,ultra thick] (a)--(e);
                \draw[black,ultra thick] (a)--(i);
                \draw[black,ultra thick] (a)--(j);
                \draw[black,ultra thick] (a)--(g);
                \node[] at (0,-3) {Some edges of $R(12,5,2)$};
                
                \node[] at ({2.5*cos(0)}, {2.5*sin(0)}) {$v_0$};;
                \node[] at ({2.5*cos(30)}, {2.5*sin(30)}) {$v_1$};;
                \node[] at ({2.5*cos(60)}, {2.5*sin(60)}) {$v_2$};;
                \node[] at ({2.5*cos(90)}, {2.5*sin(90)}) {$v_3$};;
                \node[] at ({2.5*cos(120)}, {2.5*sin(120)}) {$v_4$};;
               \node[] at ({2.5*cos(150)}, {2.5*sin(150)}) {$v_5$};;
                \node[] at ({2.5*cos(180)}, {2.5*sin(180)}) {$v_6$};;
               \node[] at ({2.5*cos(210)}, {2.5*sin(210)}) {$v_7$};;
               \node[] at ({2.5*cos(240)}, {2.5*sin(240)}) {$v_8$};;
                \node[] at ({2.5*cos(270)}, {2.5*sin(270)}) {$v_9$};;
               \node[] at ({2.5*cos(300)}, {2.5*sin(300)}) {$v_{10}$};;
                \node[] at ({2.5*cos(330)}, {2.5*sin(330)}) {$v_{11}$};;
            \end{tikzpicture}
         };

\end{tikzpicture}
\caption{The edges of $R(12,5)$ and $R(12,5,2)$ that are incident with $v_0$.}
\label{R_constructions}
\end{figure}
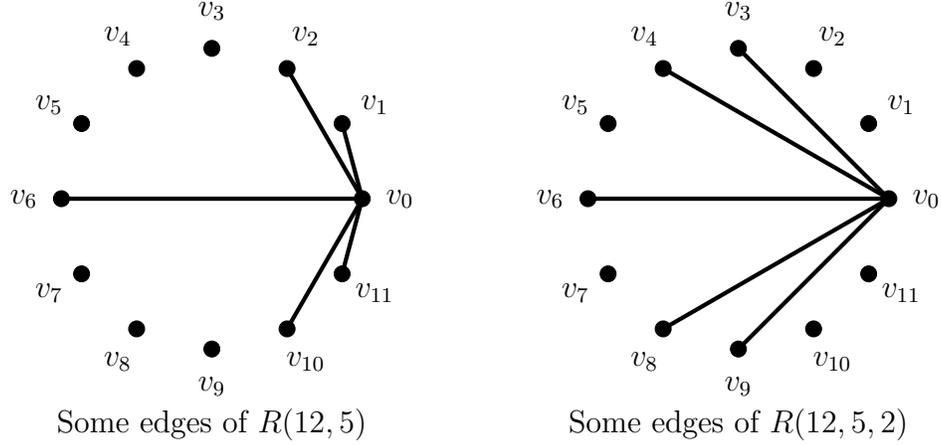

In what follows it will be useful to have upper bounds on the independence numbers of the graphs in these two families.
\begin{lemma} \label{lem-ind-no-reg-graphs}
Let $\alpha(G)$ denote the independence number of a graph $G$. We have:
\begin{equation} \label{R(m,r)bound}
\alpha(R(m,r)) \leq \lc \frac{m}{1+\lf r/2 \rf} \rc
\end{equation}
and more generally
\begin{equation} \label{R(m,r,g)bound}
\alpha(R(m,r,g)) \leq (g+1)\lc\frac{m}{g+1+\lf r/2 \rf}\rc. 
\end{equation}
\end{lemma}

\begin{proof}
We only consider \eqref{R(m,r,g)bound} since this reduces to \eqref{R(m,r)bound} when $g=0$. 

Suppose $I\subset \{v_0, \ldots, v_{m-1}\}$ is an independent set in $R(m,r,g)$. If $v_i\in I$, then $S_i:=\{v_i, \ldots, v_{i+g+\lf r/2 \rf\pmod{m}}\}$ forms a set of $1+g+\lf r/2 \rf$ consecutive vertices of which at most $g+1$ are in $I$ (specifically, $v_i, \ldots, v_{i+g \pmod{m}}$). It immediately follows that 
$$ 
\alpha(R(m,r,g))\leq(g+1)\lc\frac{m}{g+1+\lf d/2 \rf}\rc.
$$
\end{proof}

\subsection{A construction for $2n/3 \lessapprox d < n$} \label{subsection-d>2n/3-construct}

\begin{theorem} \label{thm-d>=2n/3}
Fix $n, d$ and $\alpha$ with $2n/3 \leq d < n$, $nd$ even, 
\begin{equation} \label{eq-d>=2n/3-cond}
\lc \frac{2d-n}{1+\lf \frac{3d-2n}{2}\rf} \rc \leq \alpha,
\end{equation}
and $\alpha \leq n-d$. There is an $(n,d,\alpha)$-graph $G(n,d,\alpha)$ that has $Z(n-d,\alpha)$ as an induced subgraph.  
\end{theorem}

\medskip

Note that if $d=2n/3$ then \eqref{eq-d>=2n/3-cond} becomes $\alpha \geq \lc n/3 \rc$. Since the largest independent set in an $(n,2n/3)$-graph has size at most $n/3$, it follows that for this choice of $d$ Theorem \ref{thm-d>=2n/3} gives no information. On the other hand if $d=(2/3+\Omega(1))n$ then \eqref{eq-d>=2n/3-cond} becomes $\alpha \geq \Omega(1)$, and so (given that we are considering the situation where $\alpha$ and $d$ scale linearly with $n$) Theorem \ref{thm-d>=2n/3} conveys useful information. This explains the $d \gtrapprox 2n/3$ in the section title.       

\medskip

\begin{proof} (Proof of Theorem \ref{thm-d>=2n/3}.)
We begin our construction with $G_1(n,d,\alpha)$, as described in Section \ref{subsec-base-graph-2}. Noting that 
\begin{itemize}
\item the evenness of $nd$ implies that $(2d-n)(3d-2n)=6d^2-7nd+2n^2$ is even,
\item that $d \geq 2n/3$ implies $3d-2n \geq 0$, and
\item that $d < n$ implies $2d-n > 3d-2n$,
\end{itemize}
we see that the graph $R(2d-n,3d-2n)$ exists. We add edges to $W$ so that $W$ induces a copy of $R(2d-n,3d-2n)$, and we add all edges of the form $uw$ where $u \in X \cup Y$ and $w \in W$. Call the resulting graph $G(n,d,\alpha)$.

The degree of each vertex in $X \cup Y$ is $(n-d) + (2d-n) =d$, and the degree of each vertex in $W$ is $(3d-2n) + 2(n-d) = d$, so $G(n,d,\alpha)$ is $d$-regular. The subgraphs induced by both $X$ and $Y$ are isomorphic to $Z(n-d, \alpha)$. There are no independent sets in $G(n,d,\alpha)$ that use vertices from both $X \cup Y$ and $W$, and, as observed in Section \ref{subsec-base-graph-2} there is no independent set of size greater than $\alpha$ in $X \cup Y$. So to show that $G$ is an $(n,d,\alpha)$-graph it remains to show that there is no independent set of size greater than $\alpha$ in $W$.   
By \eqref{R(m,r)bound} we have
$$
\alpha(R(2d-n,3d-2n)) \leq \lc\frac{2d-n}{1+\lf \frac{3d-2n}{2}\rf}\rc, 
$$
and this is at most $\alpha$ by \eqref{eq-d>=2n/3-cond}.
\end{proof}

Figure \ref{fig5_1} schematically illustrates the construction in the proof of Theorem \ref{thm-d>=2n/3}.

\begin{figure}[ht!]
\centering
\begin{tikzpicture}
\node[] at (0,1) {$K_a$};
\node[] at (0,3) {$K_a$};
\node[] at (0,5) {$K_b$};

\node[] at (2,1) {$K_a$};
\node[] at (2,3) {$K_a$};
\node[] at (2,5) {$K_b$};

\node[] at (5,3) {$R(2d-n,3d-2n)$};

\draw[black,ultra thick] (0.3,1)--(1.7,1);
\draw[black,ultra thick] (0.3,3)--(1.7,3);
\draw[black,ultra thick] (0.3,5)--(1.7,5);

\draw[red,ultra thick] (0.3,1)--(1.7,3);
\draw[red,ultra thick] (0.3,3)--(1.7,5);
\draw[red,ultra thick] (0.3,5)--(1.7,1);
\draw[red,ultra thick] (0.3,1)--(1.7,5);
\draw[red,ultra thick] (0.3,3)--(1.7,1);
\draw[red,ultra thick] (0.3,5)--(1.7,3);


\draw[red,ultra thick] (3.2,3)--(2.3,1);
\draw[red,ultra thick] (3.2,3)--(2.3,3);
\draw[red,ultra thick] (3.2,3)--(2.3,5);

\draw[red,ultra thick] (3.2,3)--(0.3,1);
\draw[red,ultra thick] (3.2,3)to[bend right = 45](0.3,3);
\draw[red,ultra thick] (3.2,3)--(0.3,5);



\end{tikzpicture}
\caption{Construction when $d\gtrapprox 2n/3$ (black lines represent perfect matchings and red lines represent complete bipartite graphs).}
\label{fig5_1}
\end{figure}

\subsection{A construction for $n/2 \leq d \lessapprox n$}\label{d=2n/3case}

As we will see in Section \ref{LB3proof}, the construction of this section is enough to establish {\bf LB3}. But this construction has a somewhat unsatisfying gap when $d$ is close to $n$; we therefore have included the simpler construction of Section \ref{subsection-d>2n/3-construct} to close this gap. 

\begin{theorem} \label{thm-middle-d}
Fix $n, d$ and $\alpha$ with $n/2 \leq d < n$, $nd$ even, 
\begin{equation} \label{middle-d-cond1}
\lc \frac{n-d}{\alpha} \rc + \lf \frac{2d-n}{2} \rf \leq \frac{d}{2} - 1,
\end{equation}
\begin{equation} \label{middle-d-cond2}
\left(\lc \frac{n-d}{\alpha} \rc +1\right) \lc \frac{d}{\lc \frac{n-d}{\alpha} \rc +1 + \lf \frac{2d-n}{2} \rf }\rc \leq \alpha,
\end{equation}
and $\alpha\leq n-d$. There is an $(n,d,\alpha)$-graph $G(n,d,\alpha)$ that has $Z(n-d,\alpha)$ as an induced subgraph.  \end{theorem}

\medskip

Note that as $n \rightarrow \infty$ with $\alpha, d = \Theta(n)$ (as is the situation in {\bf LB3}), the bound \eqref{middle-d-cond2} is immediate (it becomes $\alpha \geq \Omega(1)$), and \eqref{middle-d-cond1} becomes $d \leq n - o(1)$, justifying the $d \lessapprox n$ in the title of this section. 

\medskip

\begin{proof} (Proof of Theorem \ref{thm-middle-d}.)
As in Section \ref{subsection-d>2n/3-construct} we start out with $G_1(n,d,\alpha)$. We then put a complete bipartite graph between $Y$ and $W$. This ensures that all vertices in $Y$ have degree $(n-d)+(2d-n)=d$, while all vertices in $X \cup W$ have degree $n-d$. We will complete the construction by putting the edges of a $(2d-n)$-regular graph on the vertex set $X \cup W$ (a set of size $d$), that doesn't use any of the currently used edges. Specifically we we add the edges of a copy of $R(d,2d-n,\lc (n-d)/\alpha \rc)$, where the vertices of $X \cup W$ are labeled so that $v_0,\ldots,v_{|X_1|-1}$ are the vertices in $X_1$, the next $|X_2|$ indices cover the vertices in $X_2$, and so on, with the final $2d-n$ indices covering $W$.

Note that $2d-n\geq 0$ (since $d \geq n/2$) and $2d-n < d$ (since $d < n$), and that the evenness of $d(2d-n)$ follows from the evenness of $nd$, so a $(d,2d-n)$-graph exists. To specifically ensure that $R(d,2d-n,\lc (n-d)/\alpha) \rc)$ exists, we also need the condition that
$$
\lc \frac{n-d}{\alpha} \rc + \lf \frac{2d-n}{2} \rf \leq \frac{d}{2} - 1,
$$
which is \eqref{middle-d-cond1}.

Next we check that adding the edges of $R(d,2d-n,\lc (n-d)/\alpha) \rc)$ to $X \cup Y$ does not create a duplicate edge (recall that inside $X$ there are already the edges of $Z(n-d,\alpha)$). We use that in $R(d,2d-n,\lc (n-d)/\alpha) \rc)$, on either side of each vertex in the cyclic ordering $v_0, \ldots, v_{m-1}$, there is a collection of $\lc (n-d)/\alpha \rc$ consecutive vertices that the vertex is not adjacent to. So, noting that the pre-existing edges in $X$ form cliques on consecutive vertices of orders $\lf (n-d)/\alpha \rf$ and $\lc (n-d)/\alpha \rc$, we have our required condition by the choice of $g=\lc (n-d)/\alpha \rc$.

We have completed our construction of a graph $G(n,d,\alpha)$ that is $d$-regular and on $n$ vertices, and that has $Z(n-d,\alpha)$ as an induced subgraph (the subgraph induced by $Y$). 

We finally need to check that $G(n,d,\alpha)$ has no independent set of size greater than $\alpha$. As before, any independent set that includes some vertices from $Y$ has size at most $\alpha$. And any independent set drawn fully from $X \cup W$ is an independent set in $R(d,2d-n,\lc (n-d)/\alpha \rc)$, and so by \eqref{R(m,r,g)bound} has size at most
$$
\left(\lc \frac{n-d}{\alpha} \rc +1\right) \lc \frac{d}{\lc \frac{n-d}{\alpha} \rc +1 + \lf \frac{2d-n}{2} \rf }\rc.
$$
That this is at most $\alpha$ comes from \eqref{middle-d-cond2}.
\end{proof}

Figure \ref{fig5_3} schematically illustrates the construction in the proof of Theorem \ref{thm-middle-d}. 

\begin{figure}[ht!] 
\centering
\begin{tikzpicture}

\node[] at (0,0) {$K_a$};
\node[] at (2,0) {$K_a$};
\node[] at (4,0) {$K_{b}$};
\node[] at (0,2) {$K_a$};
\node[] at (2,2) {$K_a$};
\node[] at (4,2) {$K_b$};
\node[] at (6,0) {$E_{2d-n}$};

\draw[black,ultra thick] (0,0.4)--(0,1.6);
\draw[black,ultra thick] (2,0.4)--(2,1.6);
\draw[black,ultra thick] (4,0.4)--(4,1.6);

\draw[red,ultra thick] (0,0.4)--(2,1.6);
\draw[red,ultra thick] (0,0.4)--(4,1.6);
\draw[red,ultra thick] (2,0.4)--(4,1.6);
\draw[red,ultra thick] (2,0.4)--(0,1.6);
\draw[red,ultra thick] (4,0.4)--(2,1.6);
\draw[red,ultra thick] (4,0.4)--(0,1.6);

\draw[red,ultra thick] (0,1.6)--(5.3,.3);
\draw[red,ultra thick] (2,1.6)--(5.3,.3);
\draw[red,ultra thick] (4,1.6)--(5.3,.3);

\draw[blue,ultra thick] (3.1,0) circle [x radius=3.7, y radius=0.5];
\node[] at (3.1,-1) {\color{blue} $R(d,2d-n,b)$};

\end{tikzpicture}
\caption{Construction when $n/2 \leq d \lessapprox n$ (black lines represent perfect matchings and red lines represent complete bipartite graphs).}
\label{fig5_3}
\end{figure}

\subsection{Proof of {\bf LB3}}\label{LB3proof}

Here we establish {\bf LB3}. Recall that we are given $1/2 \leq c_{\rm deg} < 1$ and $0 \leq c_{\rm ind} \leq 1-c_{\rm deg}$, and that (as observed after the statement of  Theorem \ref{thm-ind-sets-alpha=cn-d=an-a>=1/2}), in the presence of Theorem \ref{thm-ind-sets-alpha=cn-d<=n/2} we may assume that $c_{\rm deg} > 1/2$. 

Let $N$ be the set of even numbers (so that for $n \in N$ we can be sure that $nd$ is always even). Let $(d_n)_{n \in N}$ be any sequence of integers such that $d_n/n \rightarrow c_{\rm deg}$ as $n \rightarrow \infty$ and let $(\alpha_n)_{n \in N}$ be any sequence of integers such that $\alpha_n/n \rightarrow c_{\rm ind}$ as $n \rightarrow \infty$). 
For all sufficiently large $n$ \eqref{middle-d-cond1} and \eqref{middle-d-cond2} both hold, and Theorem \ref{thm-middle-d} shows that there is an $(n,d_n,\alpha_n)$-graph $G_n$ with $Z(n-d_n,\alpha_n)$ as an induced subgraph. This graph satisfies 
$$
i(G_n) \geq i(Z(n-d,\alpha))=k(c_{\rm ind},c_{\rm deg})^{n+o(n)},
$$ 
the equality coming from \eqref{to-show-zykov-large-d}.

\section{Discussion and open questions} \label{sec-discussion}

We have established two approximate results. First, for each $0 < c_{\rm ind} \leq 1/2$ we have found a constant $k(c_{\rm ind})$ with the following property: the greatest number of independent sets admitted by an $n$-vertex $d$-regular graph with maximum independent set size at most $\alpha$, where $\alpha \sim c_{\rm ind}n$, $d \lesssim n/2$ and $d=\omega(1)$, is $k(c_{\rm ind})^{n+o(n)}$. Second, for each $1/2 \leq c_{\rm deg} < 1$ and $0 < c_{\rm ind} \leq 1-c_{\rm deg}$ we have found a constant $k(c_{\rm ind},c_{\rm deg})$ with the following property: the greatest number of independent sets admitted by an $n$-vertex $d$-regular graph with maximum independent set size at most $\alpha$, where $\alpha \sim c_{\rm ind}n$ and $d \sim c_{\rm deg}n$, is $k(c_{\rm ind},c_{\rm deg})^{n+o(n)}$.  

Both of these results are in the spirit of Granville's conjecture \eqref{Granville} and Alon's resolution of the same, for independent sets in $n$-vertex, $d$-regular graphs. An open problem is to obtain a more precise result, in the spirit of Kahn and Zhao's strengthening of \eqref{Granville} to the statement that if $G$ is an $n$-vertex, $d$-regular graph with $2d$ dividing $n$, then $i(G) \leq i(G_{n,d})$, where $G_{n,d}$ is the disjoint union of $n/(2d)$ copies of $K_{d,d}$.
\begin{quest}
Let $(n,d,\alpha)$ be a triple for which $n$-vertex, $d$-regular graphs with maximum independent set sizes at most $\alpha$ exist. Which such graph admits the most independent sets?
\end{quest}

There is an even more fundamental open problem in this direction. As observed earlier, the bound $i(G)\leq i(G_{n,d})$ for $n$-vertex $d$-regular graphs is tight when $2d|n$ and $d \leq n/2$; but when $d > n/2$ it is far from tight. In Claim \ref{clm-(n,d):d>n/2} we have established that 
$$
i(G) \leq 2^{n-d+\log_2 n}
$$
for $n$-vertex $d$-regular graphs $G$ with $n/2 < d < n$ and $nd$ even, and shown that there are examples of $n$-vertex $d$-regular graphs $G$ with $n/2 < d < n$ and $nd$ even that satisfy $i(G) \geq 2^{n-d}$. But it is open to identify, for each pair $(n,d)$ with $n/2 < d < n$ and $nd$ even, the specific $n$-vertex $d$-regular graph that maximizes the number of independent sets. A conjecture has recently been made by Cambie, Goedgebeur and Jooken \cite{CambieGoedgebeurJooken}:
\begin{conjecture} \label{conj-CGJ}
Let $n$ and $d$ with $nd$ even satisfy $n \geq 1$ and $d+1 \leq n \leq 2d$. Among $n$-vertex $d$-regular graphs, the graph maximizing $i(G)$ is the mutual join of copies of $E_{n-d}$ and a graph $H$ whose order is strictly less than $2(n-d)$. 
\end{conjecture}

Another direction to pursue is to consider not the total count of independent sets, but the number of independent sets of various fixed sizes. For $n$-vertex $d$-regular graphs this is a well-studied problem, with the best results to date due to Davies, Jenssen and Perkins \cite{DaviesJenssenPerkins}. The techniques of this paper can easily be used to give some results in this direction for $n$-vertex $d$-regular graphs with maximum independent set size at most $\alpha$, owing to the fact that Zykov's Theorem (Theorem \ref{thm-Zykov}) addresses independent set of each fixed size. For example, 
\begin{itemize}
\item for each $0 < c_{\rm ind} \leq 1/2$ and $0 < t \leq c_{\rm ind}$ we can exhibit an optimal constant $k(c_{\rm ind},t)$ such that the maximum number of independent sets of size $tn+o(n)$ in an $n$-vertex, $d$-regular graph with maximum independent size at most $\alpha$, where $\alpha = c_{\rm ind}n +o(n)$, is $k(c_{\rm ind},t)^{n+o(n)}$, and
\item for each $1/2 \leq c_{\rm deg} < 1$, $0 < c_{\rm ind} \leq 1-c_{\rm deg}$ and $0 < t \leq c_{\rm ind}$ we can exhibit an optimal constant $k(c_{\rm ind},c_{\rm deg},t)$ such that the maximum number of independent sets of size $tn+o(n)$ in an $n$-vertex, $d$-regular graph with maximum independent size at most $\alpha$, where $d=c_{\rm deg}n+o(n)$ and $\alpha = c_{\rm ind}n +o(n)$, is $k(c_{\rm ind},c_{\rm deg},t)^{n+o(n)}$.
\end{itemize}
We omit the details.

A final open question concerns our constructions in Section \ref{sec-constructions-d<=n/2} and \ref{sec-constructions-d>=n/2}. 
\begin{quest}
Is it the case that for every triple $(n, d, \alpha)$ for which
there exists an $n$-vertex, $d$-regular graph with maximumum independent set size at most $\alpha$, there exists one that has $Z(\lf n/2 \rf , \alpha)$ as an induced subgraph (if $d \leq n/2$) or has $Z(n - d,\alpha)$
as an induced subgraph (if $d \geq n/2$), and that has maximum independent set size at most $\alpha$?
\end{quest}
We have answered this question for enough triples $(n, d, \alpha)$ to establish {\bf LB1} and {\bf LB3}, but not for all such triples.

\end{document}